\documentclass[10pt]{article}

\usepackage{amsmath}
\usepackage{amsthm}
\usepackage{amssymb}
\usepackage{amsfonts}
\usepackage{epsf}
\usepackage{graphicx}
\usepackage{epstopdf}
\usepackage{hyperref}

\newtheorem{lem}{Lemma}[section]
\newtheorem{thm}{Theorem}[section]
\newtheorem{cor}{Corollary}[section]
\newtheorem{conj}{Conjecture}[section]
\newtheorem{prop}{Proposition}[section]
\theoremstyle{remark}
\newtheorem{rmk}{Remark}[section]
\theoremstyle{definition}

\numberwithin{equation}{section}

\DeclareMathOperator\Ai{{Ai}}
\DeclareMathOperator\Bi{{Bi}}

\DeclareMathOperator\im{{Im}}

\def\cut{\setminus}
\def\ep{\varepsilon}
\def\a{\alpha}
\def\b{\beta}
\def\la{\lambda}
\def\s{\sigma}
\def\x{\xi}
\def\G{\Gamma}
\def\P{\mathcal{P}}
\def\Q{\mathcal{Q}}

\begin{document}

\title{Plancherel-Rotach asymptotic expansion for some polynomials from indeterminate moment problems}
\date{\today}
\author{Dan Dai$^\ast $, Mourad E.H. Ismail$^\dag$ and Xiang-Sheng Wang$^\ddag$}

\maketitle

\begin{abstract}
  We study the Plancherel--Rotach asymptotics of four families of orthogonal polynomials, the Chen--Ismail polynomials, the Berg-Letessier-Valent polynomials, the Conrad--Flajolet polynomials I and II. All these polynomials arise in indeterminate moment problems and three of them are birth and death process polynomials with cubic or quartic rates. We employ a difference equation asymptotic technique due to Z. Wang and R. Wong. Our analysis leads to a conjecture about large degree behavior of polynomials orthogonal
  with respect to solutions of indeterminate moment problems.

\end{abstract}



\vspace{3mm}


\vspace{5mm}

\hrule width 65mm

\vspace{2mm}

\begin{description}
\item  \hspace*{3.8mm}$\ast $ Department of Mathematics, City University of
Hong Kong, Hong Kong SAR.
Phone:	+852 3442-5995.
Fax:	+852 3442-0250.\\
Email: \texttt{dandai@cityu.edu.hk.}

\item \hspace*{3.8mm}$\dag $  Department of Mathematics, University of Central Florida, Orlando, Florida, USA
32816 and King Saud University, Riyadh, Saudi Arabia. \\
Phone: +1 407-808-7466.
Fax: +1 407-823-6253.\\
Email: \texttt{mourad.eh.ismail@gmail.com}

\item \hspace*{3.8mm}$\ddag $   Department of Mathematics and Statistics, Memorial University of Newfoundland, St. Johns, NL A1C 5S7,
Canada.
Phone:	+1 709-864-4321.
Fax:	+1 709-864-3010. \\
Email: \texttt{xswang@mun.ca} (Corresponding author)

\end{description}

\newpage

\section{Introduction}
The Plancherel-Rotach asymptotics for orthogonal polynomials refer to asymptotics of the orthogonal polynomials scaled by the largest zero. There are three types of Plancherel--Rotach asymptotics depending on whether we are in the oscillatory region between the largest and smallest zeros, in the exponential region beyond the largest (or smallest) zero, or in the transition region near the largest (or smallest) zero. If the extreme zero moves as the polynomial degree increases, the last type of asymptotics is related to the soft edge asymptotics in the random matrix theory and describes the eigenvalues of large random Hermitian matrices near the tail end, see
\cite{And:Gui:Zei}. The Plancherel-Rotach  asymptotics for Hermite and Laguerre polynomials are given in
\cite{Sze} and \cite{And:Gui:Zei}.  The corresponding asymptotics for the exponential weights were
developed by using the Deift-Zhou steepest descent analysis for corresponding Riemann-Hilbert problems in \cite{Dei:Kri:McL:Ven:Zho} and \cite{Kri:McL}.

Every birth and death process leads to a family of orthogonal polynomials. The birth rates
$\{\la_n \}$
and the death rates $\{\mu_n\}$ are assumed to satisfy the conditions that $\la_n$ and $\mu_{n+1}$ are positive and
$\mu_0  \ge 0$. The recurrence relation of the birth and death process polynomials $\{\mathcal{Q}_n(x)\}$ is
\begin{equation} \label{bir&dea}
  -x\mathcal{Q}_n(x) = \la_n \mathcal{Q}_{n+1}(x) + \mu_n\mathcal{Q}_{n-1}(x) - (\la_n+\mu_n)\mathcal{Q}_n(x), \quad n >0
\end{equation}
with the initial conditions
 \begin{eqnarray*}
 \mathcal{Q}_0(x) =1, \mathcal{Q}_1(x) = (\la_0 +\mu_0 -x)/\la_0.
 \end{eqnarray*}
It is known in the literature that some classical continuous and discrete orthogonal polynomials can be written as birth and death process polynomials with certain special rates. For example, one can find the birth and death rates for Laguerre polynomials, Jacobi polynomials and Charlier polynomials in \cite[Sec. 5.2]{Ismbook}. Of course, for these classical polynomials, their weight functions are  well-known and unique. However, for quite a few choices of $\la_n$ and $\mu_n$, the corresponding moment problem is indeterminate. This means that the orthogonality measures of such polynomials are not unique. For example,
Letessier and Valent \cite{Let:Val} introduced polynomials associated with quartic rates
\begin{equation} \label{Berg&Val}
\lambda_n=(4n+1)(4n+2)^2(4n+3),\quad
\mu_n=(4n-1)(4n)^2(4n+1).
\end{equation}
These polynomials were also studied by Berg and Valent \cite{Ber:Val1} and we refer to them as Berg-Letessier-Valent polynomials.
Van Fossen Conrad and Flajolet \cite{Con,Con:Fla} considered the following two families with cubic rates
\begin{equation}
\label{masterI}
\lambda_n=(3n+c+1)(3n+c+2)^2, \quad \mu_n=(3n+c)^2 (3n+c+1), \quad c>0,
\end{equation}
\begin{equation}
\label{masterII}
\lambda_n=(3n+c+1)^2 (3n+c+2), \quad \mu_n=(3n+c-1) (3n+c)^2, \quad c>0;
\end{equation}
see also \cite{Gil:Leo:Val}. The moment problems associated with the above rates are all indeterminate.
For the Conrad-Flajolet polynomials this will be proved in Section 6.
In fact, the orthogonality measures of above polynomials are still unknown in the literature.

Chen and Ismail \cite{Che:Ism} considered the orthogonal polynomials arising from the following recurrence relation
\begin{equation} \label{fn3term}
  \mathcal{F}_{n+1}(x) = x \mathcal{F}_{n}(x) - 4n^2 (4n^2-1) \mathcal{F}_{n-1}(x)
\end{equation}
with $\mathcal{F}_{0}(x) = 1$ and $\mathcal{F}_{1}(x) = x$. They indicated that these polynomials are also related to the indeterminate moment problems. Moreover, they found the following family of weight functions for $\mathcal{F}_{n}(x)$
\begin{equation} \label{fn-weights}
  w_\a(x) = \frac{\sqrt{1-\a^2} [\cos(\rho \sqrt{x/2}) + \cosh(\rho \sqrt{x/2})]}{2[ \{\cos(\rho \sqrt{x/2}) + \cosh(\rho \sqrt{x/2}) \}^2 - \a^2 \sin^2(\rho \sqrt{x/2}) \sinh^2(\rho \sqrt{x/2})]}, \quad x\in \mathbb{R},
\end{equation}
where $\a \in (-1,1)$, $\rho$ is a constant defined in \eqref{rho-fn} below and $\sqrt{-1}$ is defined as $i$; cf. \cite[eq. (5.11)]{Che:Ism}.

Note that the weight function \eqref{fn-weights} has a singularity at $0$, which is similar to the Freud weight $e^{-|x|^\beta}$. One may want to use Riemann-Hilbert method to obtain asymptotic formulas for the Chen--Ismail polynomials, as has been done by Kriecherbauer and McLaughlin \cite{Kri:McL} for Freud polynomials. However, since the orthogonality measures are usually unknown or too complicated (cf. Chen--Ismail polynomials) for orthogonal polynomials related to the indeterminate moment problems, it is difficult to apply the powerful Deift-Zhou steepest descent analysis for Riemann-Hilbert problems.
In this paper, we intend to use the difference equation method developed by Wong and his colleagues in \cite{xswang2012,wang-wong2003,wang-wong2005} to get the asymptotic expansions. This will be the first time to apply this technique to derive Plancherel-Rotach asymptotics for orthogonal polynomials corresponding to indeterminate moment problems. It must be emphasized that we only use the recurrence relation (or difference equations) to achieve our asymptotic expansion in this paper.  This is a significant improvement comparing with the previous works, where some extra asymptotic information is needed to determine the asymptotic expansion; for example, see \cite{Lee:Wong,wang-wong2003,wang-wong2005}.  Because we are dealing
with polynomials orthogonal with respect to infinitely many measures it is important to use
a technique which only depends on the moments but not on the specific orthogonality measure used.

Before we proceed to  the next section, we would like to mention that one useful tool in estimating the largest and smallest zeros of orthogonal polynomials is the following theorem of Ismail and Li \cite{Ism:Li}. We did not consider any $q$-orthogonal polynomials in this work. Ismail
and Li \cite{Ism:Li2} considered the Placherel-Rotach asymptotics for symmetric
$q$-polynomials when the recursion coefficients grow exponentially. We hope
to treat these types of polynomials in a future work using difference equation
techniques.
\begin{thm}\label{thmIsmLi}
Let $\left\{\mathcal{P}_n(x)\right\}$ be a sequence of monic polynomials satisfying
\begin{equation} \label{monic-recur}
  x\mathcal{P}_n(x) = \mathcal{P}_{n+1}(x) + \alpha_n \mathcal{P}_n(x) + \beta_n \mathcal{P}_{n-1}(x),
\end{equation}
with $\beta_n>0$, for $1\le n<N$ and let $\left\{c_n\right\}$ be a chain sequence. Set
\begin{equation}
\label{eq7.2.5}
B:=\max\{x_n:0<n<N\},\quad\text{and}\quad A:=\min\{y_n:0<n<N\},
\end{equation}
where $x_n$ and $y_n$, $x_n\ge y_n$, are the roots of the equation
\begin{equation}
\label{eq7.2.6}
(x-\alpha_n) (x-\alpha_{n-1}) c_n=\beta_n,
\end{equation}
that is
\begin{equation}
\label{chain-sol}
x_n,\; y_n=\frac{1}{2}(\alpha_n+\alpha_{n-1})\pm
\frac{1}{2}\sqrt{(\alpha_n-\alpha_{n-1})^2+4\beta_n/c_n}.
\end{equation}
Then the zeros of $\mathcal{P}_N(x)$ lie in $(A, B)$.
\end{thm}
In practice, one usually choose the simple chain sequence $c_n \equiv 1/4$ for all $n$ in the above theorem.

The present paper is organized as follows. In Section 2, we study the Chen--Ismail polynomials. After applying results in Wang and Wong \cite{wang-wong2005} to obtain two linearly independent solutions for the difference equation \eqref{fn3term}, we develop ideas used in Wang and Wong \cite{xswang2012} to determine the coefficients of these two solutions. Then the Plancherel-Rotach asymptotic expansion is obtained, which is given in terms of Airy functions. Moreover, we study the limiting zero distribution as well as the behavior of the largest zero. In the next three sections, we follow the similar procedures as in Section 2 and derive asymptotic expansions for Berg-Letessier-Valent polynomials and Conrad--Flajolet polynomials, respectively. Since Bessel-type asymptotic expansions appear in these cases, we need to apply the recent results by Cao and Li \cite{Cao:Li} to get the two linearly independent solutions of our difference equations. As Sections 2-5 are independent of each other,
we shall use the same notation for the functions and variables but it will mean different things in different sections. We hope
that this will not cause any confusion. In the last section of this paper, we list some remarks about the moment problem and formulate a conjecture about large degree (Plancherel--Rotach) behavior of orthogonal polynomials associated with indeterminate moment problems. In this section, we also apply asymptotic results to show that the moment problem associated with the Conrad-Flajolet polynomials is indeterminate.

We used the term Conrad-Flajolet polynomials because the proper name ``Van Fossen Conrad--Flajolet polynomials" is just too long and we are sure that Eric Van Fossen Conrad will not mind.


\section{Chen--Ismail polynomials}

Before we derive the Plancherel-Rotach asymptotics for $\mathcal{F}_{n}(x)$, one can easily get a bound for the largest and smallest zeros from Theorem \ref{thmIsmLi}.

\begin{prop} \label{fn-chain}
  Let $x_{n,k}$ be zeros of $\mathcal{F}_{n}(x)$ such that $x_{n,1}>x_{n,2}> \cdots>x_{n,n}$, then we have the following bounds for all $n\geq 1$
  \begin{equation*}
    x_{n,1} < 8n^2 \sqrt{1-\frac{1}{4n^2}} \qquad \textrm{and} \qquad x_{n,n} > -8n^2 \sqrt{1-\frac{1}{4n^2}}.
  \end{equation*}
\end{prop}
\begin{proof}
  Recall the recurrence relation of $\mathcal{F}_{n}(x)$ in \eqref{fn3term} and choose $c_n =1/4$ in Theorem \ref{thmIsmLi}. Then the result follows.
\end{proof}

\subsection{Difference equation method}

To derive the Plancherel-Rotach asymptotics of $\mathcal{F}_{n}(x)$, we start from the recurrence relation in \eqref{fn3term} and apply Wang and Wong's difference equation method developed in \cite{wang-wong2003, wang-wong2005}. To use their results, we transform the recurrence relation \eqref{fn3term} into the following standard form they need
\begin{equation} \label{fn3term-new}
  p_{n+1}(x) - (A_n x + B_n) \, p_n(x) + p_{n-1}(x) = 0.
\end{equation}
Let
\begin{equation} \label{fnkn-def}
  K_n := 2^{4n} \; \Gamma^2\left( \frac{n+1}{2} \right)\Gamma\left( \frac{n+3/2}{2} \right) \Gamma\left( \frac{n+1/2}{2} \right)
\end{equation}
and $\mathcal{F}_{n}(x) = K_n \, p_n(x)$, then \eqref{fn3term} becomes
\begin{equation*}
  K_{n+1} \, p_{n+1}(x) = x K_n \, p_n(x) - 4n^2 (4n^2-1) K_{n-1} \, p_{n-1}(x).
\end{equation*}
Since $K_{n+1} = 4n^2 (4n^2-1) K_{n-1} $, the above recurrence equation reduces to the standard form \eqref{fn3term-new} with $A_n= \frac{K_n}{K_{n+1}}$ and $B_n=0$.
As $n \to \infty$, the recurrence coefficient $A_n$ satisfies the following expansion
\begin{equation} \label{FnAn-asy}
  A_n  \sim n^{-\theta} \sum_{s=0}^\infty \frac{\alpha_s}{n^s}
\end{equation}
with $\theta=2$ and
\begin{eqnarray}
  \alpha_0 =  \frac{1}{4}, \qquad \quad \alpha_1 =  -\frac{1}{4}. \label{fnalpha-s}
\end{eqnarray}
Since $A_n$ is of order $O(n^{-2})$ when $n$ is large, to balance the term $A_n x $ in \eqref{fn3term-new}, we introduce $x = \nu^2 t$ and $\nu = n + \tau_0$, where $\tau_0$ is a constant to be determined. Then the characteristic equation for \eqref{fn3term-new} is
\begin{equation*}
  \lambda^2- \alpha_0 t  \lambda + 1= 0
\end{equation*}
with $\alpha_0$ given in \eqref{fnalpha-s}. The roots of this equation are
\begin{equation*}
  \lambda = \frac{1}{2} \left[ \frac{t}{4} \pm \sqrt{(\frac{t}{4})^2 - 4} \right]
\end{equation*}
and they coincide when the quantity inside the above square root vanishes, that is, $t = t_\pm = \pm 8$. These points $t_\pm$ are called transition points for difference equations by Wang and Wong in \cite{wang-wong2003, wang-wong2005} because in their neighbourhood the behaviors of solutions to \eqref{fn3term-new} change dramatically.

Since the polynomials $p_{2n}(x)$ and $p_{2n+1}(x)$ are even and odd functions, respectively, let us consider the asymptotics near the large transition point $t_+=8$ only. According to the main theorem in \cite[p. 189]{wang-wong2003}, we have the following proposition.
\begin{prop}\label{prop-uniform-Chen-Ismail}
  When $n$ is large, $p_n(x)$ in \eqref{fn3term-new} can be expressed as
  \begin{equation*}
   p_n(x) = C_1(x) P_n(x) + C_2(x) Q_n(x),
  \end{equation*}
  where $C_1(x)$ and $C_2(x)$ are two $n$-independent functions. In the above formula, $P_n(x)$ and $Q_n(x)$ are two linearly independent solutions of \eqref{fn3term-new} satisfying the following Airy-type asymptotic expansions in the neighbourhood of $t_+=8$
  \begin{equation*}
  P_n(\nu^2 t) \sim  \left(\frac{64 \, U(t)}{t^2 - 64}\right)^{\frac{1}{4}} \left[ \Ai(\nu^{\frac{2}{3}} U(t)) \sum_{s=0}^\infty \frac{\tilde{A}_s(U)}{\nu^{s-\frac{1}{6}}} + \Ai'(\nu^{\frac{2}{3}} U(t)) \sum_{s=0}^\infty \frac{\tilde{B}_s(U)}{\nu^{s+\frac{1}{6}}} \right]
\end{equation*}
and
\begin{equation*}
  Q_n(\nu^2 t) \sim \left(\frac{64 \, U(t)}{t^2 - 64}\right)^{\frac{1}{4}} \left[ \Bi(\nu^{\frac{2}{3}} U(t)) \sum_{s=0}^\infty \frac{\tilde{A}_s(U)}{\nu^{s-\frac{1}{6}}} + \Bi'(\nu^{\frac{2}{3}} U(t)) \sum_{s=0}^\infty \frac{\tilde{B}_s(U)}{\nu^{s+\frac{1}{6}}} \right],
\end{equation*}
where $\nu = n+\frac{1}{2}$ and $U(t)$ is defined as
\begin{eqnarray}
  &&\hspace{-25pt}\frac{2}{3} [U(t)]^{\frac{3}{2}} = \sqrt{\frac{t}{2}} \left[ F(\frac{\pi}{2}, -1) - F(\arcsin\sqrt{\frac{8}{t}}, -1) \right] - \log \frac{t + \sqrt{t^2 - 64}}{8 }, \quad t \geq 8, \label{fnzeta-def1} \\
  &&\hspace{-25pt}\frac{2}{3} [-U(t)]^{\frac{3}{2}} = \cos^{-1} \frac{t}{8} - \frac{\sqrt{2t}}{8} \mathcal{B}_{1-\frac{t^2}{64}} (\frac{1}{2}, \frac{1}{4}) \qquad \qquad -8<t < 8. \label{fnzeta-def2}
\end{eqnarray}
Here $F(\varphi,k^2)$ is the elliptic integral of the first kind
  \begin{equation} \label{ellipticF}
    F(\varphi,k^2) = \int_0^\varphi \frac{d \theta}{\sqrt{1-k^2 \sin^2 \theta}},
  \end{equation}
 $\mathcal{B}_x(a,b)$ is the incomplete Beta function
  \begin{equation} \label{icbeta}
    \mathcal{B}_x(a,b) = \int_0^x y^{a-1} (1-y)^{b-1} dy
  \end{equation}
and the leading coefficients are given by
\begin{equation*}
  \tilde{A}_0(U) = 1, \qquad \tilde{B}_0(U) = 0.
\end{equation*}
\end{prop}
\begin{proof}
   Recall the asymptotic expansion for $A_n$ in \eqref{FnAn-asy}, the equation \eqref{fn3term-new} falls into the case $\theta \neq 0$ and $t_+ \neq 0 $ considered in \cite{wang-wong2003}. Following their approach, we choose
\begin{equation*}
  \tau_0 = - \frac{\alpha_1 t_+ }{2 \theta} = \frac{1}{2} \qquad \textrm{and} \qquad \nu = n + \tau_0.
\end{equation*}
Then our proposition follows from the main Theorem in \cite{wang-wong2003}.
\end{proof}

\begin{rmk}
  Note that the terms involving the elliptic integral $F(\varphi,k^2)$ and the incomplete Beta function $\mathcal{B}_x(a,b)$ in \eqref{fnzeta-def1} and \eqref{fnzeta-def2} can be written as
  \begin{eqnarray*}
    \sqrt{\frac{t}{2}} \left[ F(\frac{\pi}{2}, -1) - F(\arcsin\sqrt{\frac{8}{t}}, -1) \right] & = &  \int_8^t \frac{(t/s)^{\frac{1}{2}}}{\sqrt{s^2-64}} ds,  \\
    \frac{\sqrt{2t}}{8} \mathcal{B}_{1-\frac{t^2}{64}} (\frac{1}{2}, \frac{1}{4}) &=&  \int_t^8 \frac{(t/s)^{\frac{1}{2}}}{\sqrt{64-s^2}} ds.
  \end{eqnarray*}
   Moreover, one can verify that $U(t)$ defined in \eqref{fnzeta-def1} and \eqref{fnzeta-def2} is a monotonically increasing function in the neighbourhood of 8. Also, we have the following asymptotic formula
   \begin{equation} \label{FnUt-asy}
     U(t) = \frac{t-8}{8 \sqrt[3]{2}} + O(t-8)^2, \qquad \textrm{as } t \to 8.
   \end{equation}
\end{rmk}


\subsection{Determination of $C_1(x)$ and $C_2(x)$}

In this subsection, we will determine the coefficients $C_1(x)$ and $C_2(x)$ in Proposition \ref{prop-uniform-Chen-Ismail} via a matching method.
To this end, we shall derive the asymptotic formulas of $\mathcal{F}_n(x)$ (or, $p_n(x)$) in the exponential region and oscillatory region, respectively.
In the following lemma, we provide the asymptotic formula in the exponential region of the solution to a general class of difference equations.
A similar result was obtained by Van Assche and Geronimo in \cite{VanAssche:Geronimo1989}.
Here, we adopt the approach developed by Wang and Wong in \cite{xswang2012}.

\begin{lem}\label{lem-exponential-general}
  Let $\pi_n(x)$ be monic polynomials defined from the following recurrence relation
  $$\pi_{n+1}(x)=(x-a_n)\pi_n(x)-b_n\pi_{n-1}(x)$$
  with initial conditions $\pi_0(x)=1$ and $\pi_1(x)=x-a_0$. Here the constants $a_n$ and $b_n$ are assumed to be polynomials in $n$ and have the following asymptotic behaviors as $n\to\infty$:
  \begin{align*}
    a_n&= an^p+\a n^{p-1}+O(n^{p-2})\\
    b_n&= b^2n^{2p}+\b n^{2p-1}+O(n^{2p-2}),
  \end{align*}
  where $a$ and $b\ge0$ are not identically zero and $p>0$ is a positive integer. Let $I$ be the smallest convex and closed interval which contains $0$, $a-2b$ and $a+2b$, namely, $I=[a-2b,a+2b]$ if $a-2b<0<a+2b$; $I=[0,a+2b]$ if $a-2b\ge0$; $I=[a-2b,0]$ if $a+2b\le0$. Note that $b\ge0$ but $a$ could be negative. So, it is possible (although rarely) that $a+2b\le0$.
  Rescale the variable $x$ by $x=x_n:=(n+\s)^py$ with $y\in\mathbb{C}\cut I$, then we have the following asymptotic formula for $\pi_n(x_n)$
\begin{align}
  \pi_n(x_n)&\sim ({n^p\over2})^n\left[{(y-a)+\sqrt{(y-a)^2-4b^2}\over 2y}\right]^{1/2}
  \nonumber\\&\times\exp\{n\int_0^1\log[(y-a r^p)+\sqrt{(y-a r^p)^2-4b^2r^{2p}}]dr\}
  \nonumber\\&\times\exp[\int_0^1{a\over2\sqrt{(y-ar)^2-4b^2r^2}}dr+\int_0^1{4b^2r+a(y-ar)\over2[(y-ar)^2-4b^2r^2]}dr]
  \nonumber\\&\times\exp[\int_0^1{p\s y\over\sqrt{(y-as^p)^2-4b^2s^{2p}}}ds-\int_0^1{\a\over p\sqrt{(y-ar)^2-4b^2r^2}}dr]
  \nonumber\\&\times\exp[-\int_0^1{2\b r\over p\sqrt{(y-ar)^2-4b^2r^2}[(y-ar)+\sqrt{(y-ar)^2-4b^2r^2}]}dr]. \label{key-lemma}
\end{align}
\end{lem}

\begin{proof}
Let
\begin{align*}
  \pi_n(x)=\prod_{k=1}^n w_k(x),
\end{align*}
we obtain $w_1(x)=x-a_0$ and
\begin{align}\label{wk-recurrence}
  w_{k+1}(x)=x-a_k-{b_k\over w_k(x)}
\end{align}
for $k\ge1$. Since $x=x_n=(n+\s)^py$, we have from successive approximation that for large $n$,
  \begin{align}\label{wkxn}
  w_k(x_n)&={(x_n-a_k)+\sqrt{(x_n-a_k)^2-4b_k}\over2}(1+\ep_k),
\end{align}
where
\begin{align*}
  \ep_k&={a_{k+1}-a_k\over2\sqrt{(x_n-a_k)^2-4b_k}}+{2(b_{k+1}-b_k)+(x_n-a_k)(a_{k+1}-a_k)\over2[(x_n-a_k)^2-4b_k]}+O(n^{-2})\\
  &={pak^{p-1}\over2\sqrt{(yn^p-ak^p)^2-4b^2k^{2p}}}+{4pb^2k^{2p-1}+pak^{p-1}(yn^p-ak^p)\over2[(yn^p-ak^p)^2-4b^2k^{2p}]}+O(n^{-2})
\end{align*}
uniformly for $k=1,2,\cdots,n$. Here we have made use of the facts that $a_k$ and $b_k$ are polynomials in $k$, and $k/n=O(1)$ uniformly in $k=1,2,\cdots,n$.
The leading term in \eqref{wkxn} is derived from solving the characteristic equation for \eqref{wk-recurrence}.
The first-order term of $\ep_k$ can be obtained (at least formally) from a standard successive approximation.
A rigorous proof of the above asymptotic formula can be given via induction.
On the other hand, noting that $x_n\sim yn^p+\x n^{p-1}$ with $\x=p\s y$, we obtain
\begin{align*}
  (x_n-a_k)+\sqrt{(x_n-a_k)^2-4b_k}&=(yn^p+\x n^{p-1}-ak^p-\a k^{p-1})
  \\& \hspace{-8em} +\sqrt{(yn^p+\x n^{p-1}-ak^p-\a k^{p-1})^2-4b^2k^{2p}-4\b k^{2p-1}} + O(n^{p-2}).
\end{align*}
A simple calculation yields
\begin{align*}
  {(x_n-a_k)+\sqrt{(x_n-a_k)^2-4b_k}\over(yn^p-ak^p)+\sqrt{(yn^p-ak^p)^2-4b^2k^{2p}}}&= 1+\hat\ep_k,
\end{align*}
where
\begin{align*}
  \hat\ep_k&= {\x n^{p-1}-\a k^{p-1}\over\sqrt{(yn^p-ak^p)^2-4b^2k^{2p}}}
  \nonumber\\&-{2\b k^{2p-1}\over\sqrt{(yn^p-ak^p)^2-4b^2k^{2p}}[(yn^p-ak^p)+\sqrt{(yn^p-ak^p)^2-4b^2k^{2p}}]} + O(n^{-2})
\end{align*}
uniformly for $k=1,2,\cdots,n$. By trapezoidal rule, we have the following three asymptotic formulas:
\begin{align*}
  &\sum_{k=1}^n\log[(yn^p-ak^p)+\sqrt{(yn^p-ak^p)^2-4b^2k^{2p}}]
  \sim np\log n +n\int_0^1\log[(y-as^p)
  \nonumber\\
  &~+\sqrt{(y-as^p)^2-4b^2s^{2p}}]ds
  +{1\over2}\log{(y-a)+\sqrt{(y-a)^2-4b^2}\over 2y},
\end{align*}
and
\begin{align*}
  \sum_{k=1}^n\ep_k&\sim\int_0^1{pas^{p-1}\over2\sqrt{(y-as^p)^2-4b^2s^{2p}}}+{4pb^2s^{2p-1}+pas^{p-1}(y-as^p)\over2[(y-as^p)^2-4b^2s^{2p}]}ds
  \nonumber\\&=\int_0^1{a\over2\sqrt{(y-ar)^2-4b^2r^2}}+{4b^2r+a(y-ar)\over2[(y-ar)^2-4b^2r^2]}dr,
\end{align*}
and
\begin{align*}
  \sum_{k=1}^n\hat\ep_k&\sim\int_0^1{\xi-\a s^{p-1}\over\sqrt{(y-as^p)^2-4b^2s^{2p}}}ds
  \nonumber\\&-\int_0^1{2\b s^{2p-1}\over\sqrt{(y-as^p)^2-4b^2s^{2p}}[(y-as^p)+\sqrt{(y-as^p)^2-4b^2s^{2p}}]}ds
  \nonumber\\&=\int_0^1{\xi\over\sqrt{(y-as^p)^2-4b^2s^{2p}}}ds-\int_0^1{\a\over p\sqrt{(y-ar)^2-4b^2r^2}}dr
  \nonumber\\&-\int_0^1{2\b r\over p\sqrt{(y-ar)^2-4b^2r^2}[(y-ar)+\sqrt{(y-ar)^2-4b^2r^2}]}dr.
\end{align*}
Combining the above three asymptotic formulas yields the desired formula.
\end{proof}

A direct application of the above lemma yields the following result.

\begin{prop}
   As $n \to \infty$, we have
  \begin{align} \label{fn-outside}
    \mathcal{F}_{n}(8(n+1/2)^2 y) &\sim \left({2n\over e}\right)^{2n} \exp\{n\log(y + \sqrt{y^2-1}) + (2n+1) \sqrt{y} F(\arcsin(\frac{1}{\sqrt{y}}),-1)\}
      \nonumber\\&\times\left( \frac{y+\sqrt{y^2-1}}{2\sqrt{y^2-1}} \right)^{1/2}
  \end{align}
  for complex $y$ bounded away from the interval $[-1,1]$, where $F(\varphi,k^2)$ is the elliptic integral of the first kind in \eqref{ellipticF}.

\end{prop}

\begin{proof}
Choose $a=\a=0$, $b=4$, $\b=0$, $p=2$, $\s=1/2$ and replace $y$ by $8y$.
We obtain from (\ref{key-lemma})
\begin{align*}
  \mathcal{F}_{n}(x)&\sim (2n)^{2n}\exp[n\int_0^1\log(y+\sqrt{y^2-t^4})dt+\int_0^1{y\over\sqrt{y^2-t^4}}dt]\left({y+\sqrt{y^2-1}\over2\sqrt{y^2-1}}\right)^{1/2}.
\end{align*}
Note that
\begin{align*}
  \int_0^1\log(y+\sqrt{y^2-t^4})dt&=
  t\log(y+\sqrt{y^2-t^4})\bigg|_0^1+\int_0^1{2t^4\over\sqrt{y^2-t^4}(y+\sqrt{y^2-t^4})}dt
  \nonumber\\&=\log(y+\sqrt{y^2-1})+2\sqrt y\int_0^{1/\sqrt y}{s^4\over\sqrt{1-s^4}+1-s^4}ds
  \nonumber\\&=\log(y+\sqrt{y^2-1})-2+2\sqrt y\int_0^{1/\sqrt y}{1\over\sqrt{1-s^4}}ds
  \nonumber\\&=\log(y+\sqrt{y^2-1})-2+2\sqrt y F(\arcsin{1\over\sqrt y}|-1),
\end{align*}
and
\begin{align*}
  \int_0^1{y\over\sqrt{y^2-t^4}}dt=\sqrt y\int_0^{1/\sqrt y}{1\over\sqrt{1-s^4}}ds=\sqrt y F(\arcsin{1\over\sqrt y},-1).
\end{align*}
Combining the above three equations yields (\ref{fn-outside}).
\end{proof}

Once we have the asymptotic formula for $\mathcal{F}_n(x)$ outside the interval where the zeros are located, using an argument similar to that in the proof of \cite[Theorem 3.2]{xswang2012}, we obtain the following result in the oscillatory region. The main idea of this argument is to match the asymptotic formula in the outer region with that in the oscillatory region. Note that the second-order recurrence relation \eqref{bir&dea} has two linearly independent solutions in the oscillatory region. The unique solution with given initial conditions can be written in terms of a linear combination of these two solutions, while the coefficients can be determined via asymptotic matching.
Due to the symmetry of $\mathcal{F}_{n}(x)$, we consider positive $x$ only.
\begin{prop}
  Let $\delta >0$ be any fixed small number. As $n \to \infty$, we have
  \begin{align}
      \mathcal{F}_{n}(8(n+1/2)^2 \cos \theta) & \sim (2n)^{2n} e^{-2n} \frac{e^{(n+\frac{1}{2}) \rho \sqrt{\cos \theta}}}{\sin^{\frac{1}{2}} \theta} \biggl[ \cos\biggl( (n+\frac{1}{2})(\theta - \frac{\sqrt{\cos \theta}}{2} \mathcal{B}_{\sin^2\theta}(\frac{1}{2}, \frac{1}{4}) ) \biggr) \nonumber  \\
    & \qquad \qquad  + \sin\biggl( (n+\frac{1}{2})(\theta- \frac{\sqrt{\cos\theta}}{2} \mathcal{B}_{\sin^2\theta}(\frac{1}{2}, \frac{1}{4}) \biggr) \biggr]  \label{fn-inside}
  \end{align}
  for $\theta\in [\delta, \frac{\pi}{2}-\delta]$, where $\mathcal{B}_x(a,b)$ is the incomplete Beta function in \eqref{icbeta} and the constant $\rho$ is
  \begin{equation} \label{rho-fn}
    \rho = 2\int_0^1{1\over\sqrt{1-s^4}}ds =  2 F(\frac{\pi}{2},-1)  = \frac{2 \sqrt{\pi} \Gamma(\frac{5}{4})}{\Gamma(\frac{3}{4})}.
  \end{equation}
\end{prop}

Finally, we obtain our theorem
\begin{thm}
  Let $\nu = n+\frac{1}{2}$. With $K_n$ and $U(t)$ defined in \eqref{fnkn-def}, \eqref{fnzeta-def1} and \eqref{fnzeta-def2}, respectively, we have
  \begin{eqnarray} \label{fn-main-for}
    \mathcal{F}_{n}(\nu^2 \, t) &\sim&  \frac{K_n}{4\sqrt{2}\, \pi^{3/2}} \exp\left( \frac{\Gamma(\frac{5}{4})}{\Gamma(\frac{3}{4})} \sqrt{\frac{\pi  t}{2}} \,\nu \right) \left(\frac{64 \, U(t)}{t^2 - 64}\right)^{\frac{1}{4}} \nonumber \\
     && \times \left[ \Ai(\nu^{\frac{2}{3}} U(t)) \sum_{s=0}^\infty \frac{\tilde{A}_s(U)}{\nu^{s-\frac{1}{6}}} + \Ai'(\nu^{\frac{2}{3}} U(t)) \sum_{s=0}^\infty \frac{\tilde{B}_s(U)}{\nu^{s+\frac{1}{6}}} \right],
  \end{eqnarray}
  uniformly for $t$ in compact subsets of  $(0,\infty)$.
\end{thm}

\begin{proof}

Let $y = \frac{(n+\frac{1}{2})^2 \, t}{8n^2}$, then $8n^2 y = \nu^2 t = x$. Because
\begin{equation*}
  n \int_y^1 \frac{\sqrt{y} }{\sqrt{u^3-u}} du \sim (n+\frac{1}{2}) \int_t^8 \frac{(t/s)^{\frac{1}{2}}}{\sqrt{64 - s^2}} ds - \frac{t}{\sqrt{64-t^2}},
\end{equation*}
we have from \eqref{fnkn-def} and \eqref{fn-inside}
\begin{eqnarray}
  p_n(x) &=& \frac{1}{K_n} \mathcal{F}_{n}(x) \sim \frac{(2n)^{2n} e^{-2n}}{K_n} \frac{\sqrt{8} \, e^{(n+\frac{1}{2}) \rho \sqrt{t/8}} }{(64-t^2)^{\frac{1}{4}}} \nonumber \\
  & & \times \biggl[ \cos\biggl( (n+\frac{1}{2})\cos^{-1}\frac{t}{8} - (n+\frac{1}{2}) \int_t^8 \frac{(t/s)^{\frac{1}{2}}}{\sqrt{64 - s^2}} ds \biggr) \nonumber  \\
  & & + \sin\biggl( (n+\frac{1}{2})\cos^{-1}\frac{t}{8} - (n+\frac{1}{2}) \int_t^8 \frac{(t/s)^{\frac{1}{2}}}{\sqrt{64 - s^2}} ds \biggr) \biggr]. \label{fnpn-com1}
\end{eqnarray}
Recall the asymptotic expansions for Airy function, as $x \to \infty$
\begin{eqnarray}
  \Ai(-x) & \sim & \frac{1}{\sqrt{\pi} x^{1/4}}  \cos ( \frac{2}{3} x^{\frac{3}{2}}-\frac{\pi}{4}), \label{Airy-asy1} \\
  \Bi(-x) & \sim & -\frac{1}{\sqrt{\pi} x^{1/4}}  \sin ( \frac{2}{3} x^{\frac{3}{2}} -\frac{\pi}{4}), \label{Airy-asy2}
\end{eqnarray}
then, we have when $t<8$
\begin{eqnarray}
  p_n(x) &=&  C_1(x) P_n(x) + C_2(x) Q_n(x) \sim C_1(x)  \left(\frac{64 }{64 - t^2 }\right)^{\frac{1}{4}}
   \frac{1}{\sqrt{\pi} }  \cos \biggl( \frac{2\nu}{3} (-U(t))^{\frac{3}{2}} -\frac{\pi}{4}\biggr) \nonumber \\
   && \quad - C_2(x) \left(\frac{64 }{64 - t^2 }\right)^{\frac{1}{4}} \frac{1}{\sqrt{\pi} }  \sin \biggl( \frac{2\nu}{3} (-U(t))^{\frac{3}{2}} -\frac{\pi}{4} \biggr) \nonumber \\
   &=& C_1(x)  \left(\frac{64 }{64 - t^2 }\right)^{\frac{1}{4}}
   \frac{1}{\sqrt{\pi} }  \cos \biggl( \nu(\cos^{-1} \frac{t}{8} - \int_t^8 \frac{(t/s)^{\frac{1}{2}}}{\sqrt{64 - s^2}} ds)  -\frac{\pi}{4}\biggr) \nonumber \\
   &&  \quad - C_2(x) \left(\frac{64 }{64 - t^2 }\right)^{\frac{1}{4}} \frac{1}{\sqrt{\pi} }   \sin \biggl( \nu(\cos^{-1} \frac{t}{8} - \int_t^8 \frac{(t/s)^{\frac{1}{2}}}{\sqrt{64 - s^2}} ds)  -\frac{\pi}{4}\biggr). \label{fnpn-com2}
\end{eqnarray}
Comparing \eqref{fnpn-com1} and \eqref{fnpn-com2}, we have
\begin{equation*}
  C_1(x) = \frac{1}{4\sqrt{2}\, \pi^{3/2}} \exp\left( \frac{\Gamma(\frac{5}{4})}{\Gamma(\frac{3}{4})} \sqrt{\frac{\pi x}{2}} \right),  \qquad C_2(x) = 0.
\end{equation*}
This completes the proof of the theorem.
\end{proof}

\begin{rmk}
Recall the weight functions for $\mathcal{F}_{n}(x)$ given in \eqref{fn-weights}. Let $\a = 0$, then we have the weight function
\begin{equation} \label{fn-weight}
  w(x) = \frac{1}{2[ \cos(\rho \sqrt{x/2}) + \cosh(\rho \sqrt{x/2})]}, \qquad x\in \mathbb{R}.
\end{equation}
When $x$ is a large number as in \eqref{fn-main-for}, that is $x=\nu^2 t$, then we have
\begin{equation*}
  w(x)^{-\frac{1}{2}} \sim \exp\left( \frac{\rho}{2} \sqrt{\frac{x}{2}} \right) \sim \exp\left( \frac{\Gamma(\frac{5}{4})}{\Gamma(\frac{3}{4})} \sqrt{\frac{\pi  t}{2}} \,\nu \right).
\end{equation*}
So the main formula in the above theorem can be stated as
    \begin{equation} \label{fn-main-rmk}
    w(x)^{\frac{1}{2}} \mathcal{F}_{n}(x) =  \frac{K_n}{4 \sqrt{2} \, \pi^{3/2}} \left(\frac{64 \, U(t)}{t^2 - 64}\right)^{\frac{1}{4}}  \left[ \Ai(\nu^{\frac{2}{3}} U(t)) \sum_{s=0}^\infty \frac{\tilde{A}_s(U)}{\nu^{s-\frac{1}{6}}} + \Ai'(\nu^{\frac{2}{3}} U(t)) \sum_{s=0}^\infty \frac{\tilde{B}_s(U)}{\nu^{s+\frac{1}{6}}} \right].
  \end{equation}
\end{rmk}

\begin{rmk} \label{fn-lim-measure}
  In fact, the three-term recurrence relation contains a lot of useful information. It can also give us the asymptotic zero distribution of the rescaled polynomials $\mathcal{F}_n(8n^2 x)$ as $n \to \infty$. Using the method developed in Kuijlaars and Van Assche \cite{Kui:Ass}, one can obtain the following limiting zero distribution for $\mathcal{F}_n(8n^2 x)$ from the recurrence relation \eqref{fn3term}
  \begin{equation*}
    \frac{1}{\pi} \int_{\sqrt{|x|}}^1 \frac{1}{\sqrt{s^4-x^2}} ds, \qquad x \in [-1,1].
  \end{equation*}
\end{rmk}

Since the above theorem is uniformly valid in the neighbourhood of the large transition point 8, we have the following asymptotic formula for $\mathcal{F}_{n}(x)$.

\begin{cor} \label{fn-cor1}
  Let $\nu = n+\frac{1}{2}$, $K_n$ and $w(x)$ be given in \eqref{fnkn-def} and \eqref{fn-weight}, respectively. Uniformly for a bounded real number $s$, we have, for $x = 8 \nu^2 + 8 \sqrt[3]{2} \,s  \,\nu^{\frac{4}{3}}$
  \begin{equation} \label{fn-airy}
    w(x)^{\frac{1}{2}} \mathcal{F}_{n}(x) =  \frac{K_n \nu^{\frac{1}{6}}}{2^{17/6} \, \pi^{3/2}}  \left[ \Ai(s)  + O(\nu^{-\frac{2}{3}}) \right], \qquad \textrm{as $\nu \to \infty.$}
  \end{equation}
\end{cor}
\begin{proof}
  Let $t = 8 + 8 \sqrt[3]{2} \,s \nu^{-\frac{2}{3}}$ in \eqref{fn-main-rmk} and recall the asymptotic formula for $U(t)$ in \eqref{FnUt-asy}, we obtain after some computations
  \begin{equation*}
    w(x)^{\frac{1}{2}} \mathcal{F}_{n}(x) =  \frac{K_n}{4 \sqrt{2} \, \pi^{3/2}} \left( 2^{-1/3} + O(\nu^{-\frac{2}{3}}) \right) \nu^{\frac{1}{6}} \left[ \Ai(s)  + O(\nu^{-\frac{2}{3}}) \right].
  \end{equation*}
  This above formula proves our corollary.
\end{proof}
\begin{rmk} \label{fn-hat-rmk}
  We may rewrite the polynomial on the left hand side of \eqref{fn-airy} into an orthonormal one. Note that if $\mathcal{P}_n(x)$ satisfies a three-term recurrence relation in \eqref{monic-recur}, then $\widehat {\mathcal{P}}_n(x):=(\prod_{k=1}^n \beta_k)^{-\frac{1}{2}}\mathcal{P}_n(x)$ is the corresponding orthonormal polynomial and satisfies the following recurrence relation
  \begin{equation*}
    x \widehat {\mathcal{P}}_n(x) = \sqrt{\beta_{{n+1}}} \widehat {\mathcal{P}}_{n+1}(x) + \alpha_n \widehat {\mathcal{P}}_n(x) + \sqrt{\beta_n} \widehat {\mathcal{P}}_{n-1}(x).
  \end{equation*}
  Let $\widehat {\mathcal{F}}_{n}(x)$ be the orthonormal Chen--Ismail polynomials, then \eqref{fn-airy} can be rewritten as
  \begin{equation} \label{fn-hat-airy}
    w(x)^{\frac{1}{2}} \widehat{\mathcal{F}}_{n}(x) =  \frac{\nu^{-\frac{5}{6}}}{2^{7/3} }  \left[ \Ai(s)  + O(\nu^{-\frac{2}{3}}) \right], \qquad \textrm{as $\nu \to \infty.$}
  \end{equation}
\end{rmk}

From the above corollary, an asymptotic approximation of extreme zeros of $\mathcal{F}_{n}(x)$ can be easily obtained by using the zeros of the Airy function $\Ai(x)$. To get it, we need the following result of Hethcote \cite{Heth}.
\begin{lem} \label{heth-lem}
  In the interval $[a-\rho, a + \rho ]$, suppose that $f (t )= g (t )+\varepsilon(t )$, where $f (t )$
is continuous, $g(t )$ is differentiable, $g (a )=0$, $m = \min|g'(t )| > 0$ and
\begin{equation*}
  E(t) = \max |\varepsilon (t )| < \min{| g (a-\rho )|, |g (a+\rho )|} .
\end{equation*}
Then, there exists a zero $c$ of $f (t )$ in the interval such that $|c-a |\leq \frac{E}{m}$.
\end{lem}
We have the following approximation.
\begin{cor}
  Let $x_{n,k}$ be the zeros of $\mathcal{F}_{n}(x)$ such that $x_{n,1}>x_{n,2}> \cdots>x_{n,n}$, and $\mathfrak{a}_k$ be the zeros of the Airy function $\Ai(x)$ in the descending order. Then we have for fixed $k$ and large $n$
  \begin{equation*}
    x_{n,k} = 8\nu^2 + 8 \sqrt[3]{2} \; \mathfrak{a}_k \nu^{4/3} + O(\nu^{2/3}),
  \end{equation*}
   where $\nu = n + \frac{1}{2}$.
\end{cor}
\begin{proof}
  A combination of Corollary \ref{fn-cor1} and Lemma \ref{heth-lem} immediately gives us the result.
\end{proof}


\section{Berg-Letessier-Valent polynomials}

Now we are going to study some birth and death process polynomials $\mathcal{Q}_n(x)$. Before we derive the Plancherel-Rotach asymptotics for Berg-Letessier-Valent polynomials $\mathcal{Q}_{n}(x)$ with the rates in \eqref{Berg&Val}, we can use the chain sequence method to get the bounds for the largest and smallest zeros as we did in Proposition \ref{fn-chain}.

\begin{prop}
  Let $x_{n,k}$ be zeros of $\mathcal{Q}_{n}(x)$ such that $x_{n,1}>x_{n,2}> \cdots>x_{n,n}$, then we have the following bounds for all $n\geq 1$
  \begin{equation*}
    x_{n,1} < 2^{10}n^4-2^{10}n^{3} + 35\cdot 2^6 n^2 \qquad \textrm{and} \qquad x_{n,n} > 4.29.
  \end{equation*}
\end{prop}
\begin{proof}
  Recall the recurrence relation of $\mathcal{Q}_{n}(x)$ in \eqref{bir&dea} and the recurrence coefficients in \eqref{Berg&Val}. We choose $c_n =1/4$ in Theorem \ref{thmIsmLi}. Note that the solutions $x_n$ and $y_n$ in \eqref{chain-sol} are monotonically increasing and decreasing with $n$, respectively. Then the result follows.
\end{proof}

\subsection{Difference equation method}

Like what we have done in the previous section, we introduce $\mathcal{Q}_n(x) = (-1)^n K_n p_n(x)$ with
\begin{equation} \label{BVkn-def}
  K_n:= \frac{\Gamma(\frac{n+\frac{3}{4}}{2}) \, \Gamma(\frac{n+1}{2})^2}{\Gamma(\frac{n+\frac{3}{2}}{2})^2 \, \Gamma(\frac{n+\frac{7}{4}}{2})}
\end{equation}
to arrive at the standard form. Note that
\begin{equation*}
  \frac{K_{n+1}}{K_{n-1}} = \frac{(4n-1)(4n)^2}{(4n+2)^2(4n+3)},
\end{equation*}
then the recurrence relation \eqref{bir&dea} with \eqref{Berg&Val} reduces to the following standard form
\begin{equation} \label{BV3term-new}
  p_{n+1}(x) - (A_n x + B_n) \, p_n(x) + p_{n-1}(x) = 0.
\end{equation}
As $n \to \infty$, the recurrence coefficients $A_n$ and $B_n$ satisfy the following expansions
\begin{eqnarray}
  A_n = \frac{K_n}{\lambda_n \cdot K_{n+1}} \sim  n^{-\theta} \sum_{s=0}^\infty \frac{\alpha_s}{n^s}, \qquad
  B_n = -\frac{(\lambda_n + \mu_n)K_n}{\lambda_n \cdot K_{n+1}} \sim \sum_{s=0}^\infty \frac{\beta_s}{n^s} \label{BVABn-asy}
\end{eqnarray}
with $\theta=4$,
\begin{eqnarray}
  \alpha_0 =  \frac{1}{256}, \qquad \quad \alpha_1 =  -\frac{1}{256} \label{BValpha-s}
\end{eqnarray}
and
\begin{eqnarray}
  \beta_0 =  -2, \qquad \quad \beta_1 = \beta_2 = 0. \label{BVbeta-s}
\end{eqnarray}
Let us introduce $x=(n+1/4)^4t$, then the characteristic equation for \eqref{BV3term-new} is
\begin{equation} \label{chara-eqn}
  \lambda^2 - (\alpha_0t + \beta_0)\lambda + 1 =0
\end{equation}
with $\alpha_0$ and $\beta_0$ given in \eqref{BValpha-s} and \eqref{BVbeta-s}, respectively. The roots of this equation coincide when
$\alpha_0 t_{\pm} + \beta_0 = \pm 2$, which gives us two transition points
\begin{equation*}
  t_+ = 2^{10}, \qquad t_- = 0.
\end{equation*}
Near the large transition point $t_+$, we get the Airy-type asymptotic expansion as in the previous section. But since the small transition point $t_-$ is located at the origin, we obtain the Bessel-type (not Airy-type) expansion in its neighbourhood. This is similar to the case of Laguerre polynomials, in which case Bessel asymptotics are obtained near the origin. Although we do not have the weight functions for Berg-Letessier-Valent polynomials, we guess they are supported on $\mathbb{R}^+$. For a detailed explanation why the case $t_-=0$ is so special to give us the Bessel type asymptotics, one may refer to discussions in \cite{Cao:Li}. We have the following two different types of expansions.

\begin{prop}
  When $n$ is large, $p_n(x)$ in \eqref{BV3term-new} can be expressed as
  \begin{equation*}
   p_n(x) = C_1(x) P_n(x) + C_2(x) Q_n(x),
  \end{equation*}
  where $C_1(x)$ and $C_2(x)$ are two $n$-independent functions, $P_n(x)$ and $Q_n(x)$ are two linearly independent solutions of \eqref{BV3term-new} satisfying the following Airy-type asymptotic expansions in the neighbourhood of $t_+=2^{10}$
  \begin{equation} \label{BV-Pn}
  P_n(\nu^4 t) \sim  2^{9/2}\left(\frac{ U(t)}{t(t -2^{10})}\right)^{\frac{1}{4}} \left[ \Ai(\nu^{\frac{2}{3}} U(t)) \sum_{s=0}^\infty \frac{\tilde{A}_s(U)}{\nu^{s-\frac{1}{6}}} + \Ai'(\nu^{\frac{2}{3}} U(t)) \sum_{s=0}^\infty \frac{\tilde{B}_s(U)}{\nu^{s+\frac{1}{6}}} \right]
\end{equation}
and
\begin{equation} \label{BV-Qn}
  Q_n(\nu^4 t) \sim 2^{9/2}\left(\frac{ U(t)}{t(t -2^{10})}\right)^{\frac{1}{4}}  \left[ \Bi(\nu^{\frac{2}{3}} U(t)) \sum_{s=0}^\infty \frac{\tilde{A}_s(U)}{\nu^{s-\frac{1}{6}}} + \Bi'(\nu^{\frac{2}{3}} U(t)) \sum_{s=0}^\infty \frac{\tilde{B}_s(U)}{\nu^{s+\frac{1}{6}}} \right].
\end{equation}
Here $\nu = n + \frac{1}{4}$, the leading coefficients are given by
\begin{equation*}
  \tilde{A}_0(U) = 1, \qquad \tilde{B}_0(U) = 0
\end{equation*}
and $U(t)$ is defined as
\begin{eqnarray}
  \frac{2}{3} [U(t)]^{\frac{3}{2}} &= & \frac{\sqrt{2} \, t^{\frac{1}{4}}}{8} \; \mathcal{B}_{1-\frac{2^{10}}{t}} (\frac{1}{2}, \frac{1}{4})  - \log \frac{t - 2^9 + \sqrt{t(t - 2^{10})}}{2^9 }, \qquad t \geq 2^{10} \label{BVzeta-def1} \\
  \frac{2}{3} [-U(t)]^{\frac{3}{2}} &=& \cos^{-1} \left( \frac{t-2^9}{2^9} \right) - \frac{\sqrt{2} \, t^{\frac{1}{4}}}{8} \; \mathcal{B}_{1-\frac{t}{2^{10}}} (\frac{1}{2}, \frac{1}{4}) \qquad \qquad 0<t < 2^{10}; \label{BVzeta-def2}
\end{eqnarray}
see \eqref{icbeta} for the definition of $\mathcal{B}_x(a,b)$.
\end{prop}
\begin{proof}
  With
  \begin{equation*}
  \tau_0 = -\frac{\alpha_1t_+ + \beta_1}{(2-\beta_0) \theta} = \frac{1}{4} \qquad \textrm{and} \qquad \nu = n + \tau_0,
  \end{equation*}
 our results follow from the main Theorem in \cite{wang-wong2003}.
\end{proof}

\begin{prop}
  When $n$ is large, $p_n(x)$ in \eqref{BV3term-new} can be expressed as
  \begin{equation*}
  p_n(x) = C_1^*(x)(-1)^n P_n^*(x) + C_2^*(x) (-1)^n Q_n^*(x)
  \end{equation*}
  where $C_1^*(x)$ and $C_2^*(x)$ are two $n$-independent functions, $P_n^*(x)$ and $Q_n^*(x)$ are two linearly independent solutions of \eqref{BV3term-new} satisfying the following Bessel-type asymptotic expansions in the neighbourhood of $t_-=0$
  \begin{equation} \label{BV-Pn2}
  P_n^*(\nu^4 t) \sim  2^{9/2}\nu^{\frac{1}{2}}\left(\frac{ U^*(t)}{t(2^{10}-t)}\right)^{\frac{1}{4}} \left[ J_{\frac{1}{2}}(\nu {U^*}^{\frac{1}{2}}(t)) \sum_{s=0}^\infty \frac{\tilde{A}_s^*(U^*)}{\nu^{s}} + J_{\frac{3}{2}}(\nu {U^*}^{\frac{1}{2}}(t)) \sum_{s=0}^\infty \frac{\tilde{B}_s^*(U^*)}{\nu^{s}} \right]
\end{equation}
and
\begin{equation} \label{BV-Qn2}
  Q_n^*(\nu^4 t) \sim 2^{9/2}\nu^{\frac{1}{2}}\left(\frac{ U^*(t)}{t(2^{10}-t)}\right)^{\frac{1}{4}} \left[ W_{\frac{1}{2}}(\nu {U^*}^{\frac{1}{2}}(t)) \sum_{s=0}^\infty \frac{\tilde{A}_s^*(U^*)}{\nu^{s}} + W_{\frac{3}{2}}(\nu {U^*}^{\frac{1}{2}}(t)) \sum_{s=0}^\infty \frac{\tilde{B}_s^*(U^*)}{\nu^{s}} \right].
\end{equation}
Here $\nu = n + \frac{1}{4}$,
\begin{equation} \label{w-alpha-def}
  W_\alpha(x) := Y_\alpha(x) - i J_\alpha (x),
\end{equation}
the leading coefficients are given by
\begin{equation*}
  \tilde{A}_0^*(U^*) = 1, \qquad \tilde{B}_0^*(U^*) = 0
\end{equation*}
and $U^*(t)$ is defined as
\begin{eqnarray}
  [-U^*(t)]^{\frac{1}{2}} &= &  \int_{t}^0 \frac{(t/s)^{\frac{1}{4}}}{\sqrt{s(s - 2^{10})}} ds-\log \frac{2^9 - t+ \sqrt{t(t - 2^{10})}}{2^9 }   , \quad t \leq 0 \label{BVzeta2-def1} \\
   \ [U^*(t)]^{\frac{1}{2}}&=& \frac{\sqrt{2} \, t^{\frac{1}{4}}}{8} \; \mathcal{B}_{\frac{t}{2^{10}}} (\frac{1}{4}, \frac{1}{2}) -\cos^{-1} \left( \frac{2^9-t}{2^9} \right)  \qquad \qquad 0<t < 2^{10}; \label{BVzeta2-def2}
\end{eqnarray}
see \eqref{icbeta} for the definition of $\mathcal{B}_x(a,b)$.
\end{prop}
\begin{proof}
  With the asymptotic expansions for $A_n$ and $B_n$ in \eqref{BVABn-asy}, our results follow from the main Theorem in \cite{Cao:Li}.
\end{proof}
\begin{rmk}
  Note that the integral in \eqref{BVzeta2-def1} can be written as a hypergeometric function as follows:
  \begin{equation*}
    \int_{t}^0 \frac{(t/s)^{\frac{1}{4}}}{\sqrt{s(s - 2^{10})}} ds=\frac{\sqrt{2} \, (-t)^{\frac{1}{2}}}{2} \; {}_2F_1(\frac{1}{4},\frac{1}{2},\frac{5}{4},t).
  \end{equation*}
  The two functions $U(t)$ and $U^*(t)$ are both monotonically increasing functions in the neighborhood of $2^{10}$ and 0, respectively. In fact, from their definitions in the above two propositions, we have the following asymptotic formulas
  \begin{eqnarray} \label{BVUt-asy}
    U(t)
    = \frac{t-2^{10}}{2^{10} \sqrt[3]{4}} + O(t-2^{10})^2, \qquad \textrm{as } t \to 2^{10}.
  \end{eqnarray}
  and
  \begin{equation*}
    U^*(t) = \frac{t}{2^8} + O(t^2) \qquad \textrm{as } t \to 0.
  \end{equation*}
\end{rmk}


\subsection{Determination of $C_1(x)$ and $C_2(x)$}

For convenience, we state a special case of Lemma \ref{lem-exponential-general} which will be used in determining $C_1(x)$ and $C_2(x)$ for all three types of birth and death polynomials considered in this paper.
\begin{lem} \label{lemma-birth}
Assume $\pi_n(x)$ satisfies the following recurrence relation
\begin{align*}
  \pi_{n+1}(x)=[x-(\la_n+\mu_n)]\pi_n(x)-\la_{n-1}\mu_n\pi_{n-1}(x)
\end{align*}
with $\pi_0(x)=1$ and $\pi_1(x)=x-(\la_0+\mu_0)$. Here, $\la_n$ and $\mu_n$ are assumed to be polynomials in $n$ and satisfy the following asymptotic formulas as $n\to\infty$:
\begin{align*}
  \la_n&= b(n^p+un^{p-1})+O(n^{p-2})\\
  \mu_n&= b(n^p+vn^{p-1})+O(n^{p-2}),
\end{align*}
where $b>0$ and $p>0$. Rescale the variable $x$ by $x=x_n:=(n+\s)^py$ with $y\in\mathbb{C}\cut[0,4b]$, then we obtain the following asymptotic formula
\begin{align}\label{pi-outer-general}
  \pi_n(x_n)&\sim ({yn^p\over e^p})^{n}(1-4b/y)^{-1/4}\left[{1+\sqrt{1-4b/y}\over 2}\right]^{2n+{u+v\over p}}
  \times\exp\{\int_0^1{(n+\s)p\over\sqrt{1-4bs^p/y}}ds\}.
\end{align}
\end{lem}

\begin{proof}
  We will show that \eqref{pi-outer-general} is a special case of \eqref{key-lemma} with $a=2b$, $\a=b(u+v)$ and $\b=b^2(u+v-p)$.
Due to the fact that $a=2b$, we have following formulas for the integrals in \eqref{key-lemma}
$$\int_0^1{a\over2\sqrt{(y-ar)^2-4b^2r^2}}dr=\int_0^1{a\over2\sqrt{y^2-2ayr}}dr={1-\sqrt{1-2a/y}\over2},$$
and
$$\int_0^1{4b^2r+a(y-ar)\over2[(y-ar)^2-4b^2r^2]}dr=\int_0^1{ay\over2[y^2-2ayr]}dr={1\over4}\log{y\over y-2a},$$
and
$$\int_0^1{\a\over p\sqrt{(y-ar)^2-4b^2r^2}}dr={\a(1-\sqrt{1-2a/y})\over pa},$$
and
\begin{align*}
  \int_0^1{2\b r\over p\sqrt{(y-ar)^2-4b^2r^2}[(y-ar)+\sqrt{(y-ar)^2-4b^2r^2}]}dr
  \\={4\b\over pa^2}[\log{2\over1+\sqrt{1-2a/y}}-{1-\sqrt{1-2a/y}\over2}].
\end{align*}
Therefore, we have from \eqref{key-lemma}
\begin{align*}
  \pi_n(x_n)&\sim ({n^p\over2})^n(1-2a/y)^{-1/4}\left[{1+\sqrt{1-2a/y}\over 2}\right]^{1+{4\b\over pa^2}}
  \nonumber\\&\times\exp\{n\int_0^1\log[(y-as^p)+\sqrt{y^2-2ays^p}]ds+\int_0^1{p\s y\over\sqrt{y^2-2ays^p}}ds\}
  \nonumber\\&\times\exp[({1\over2}-{\a\over pa}+{2\b\over pa^2})(1-\sqrt{1-2a/y})].
\end{align*}
Furthermore, since
$$(y-as^p)+\sqrt{y^2-2ays^p}=(\sqrt y+\sqrt{y-2as^p})^2/2,$$
we have
$$\int_0^1\log[(y-as^p)+\sqrt{y^2-2ays^p}]ds=-\log2+2\int_0^1\log[\sqrt y+\sqrt{y-2as^p}]ds.$$
According to integration by parts, the right-hand side becomes
\begin{align*}
  &-\log2+2\log[\sqrt y+\sqrt{y-2a}]+p\int_0^1{\sqrt y-\sqrt{y-2as^p}\over\sqrt{y-2as^p}}ds
  \\&=\log{(\sqrt y+\sqrt{y-2a})^2\over2}-p+\int_0^1{p\sqrt y\over\sqrt{y-2as^p}}ds.
\end{align*}
Therefore, we obtain
\begin{align*}
  \pi_n(x_n)&\sim ({yn^p\over e^p})^{n}(1-2a/y)^{-1/4}\left[{1+\sqrt{1-2a/y}\over 2}\right]^{2n+1+{4\b\over pa^2}}
  \nonumber\\&\times\exp\{\int_0^1{(n+\s)p\over\sqrt{1-2as^p/y}}ds+({1\over2}-{\a\over pa}+{2\b\over pa^2})(1-\sqrt{1-2a/y})\}.
\end{align*}
Note that
$${1\over2}-{\a\over pa}+{2\b\over pa^2}={1\over2}-{u+v\over 2p}+{u+v-p\over 2p}=0.$$
The asymptotic formula \eqref{pi-outer-general} for $\pi_n(x_n)$ follows.
\end{proof}

Since the birth and death process polynomials $\mathcal{Q}_n(x)$ defined in \eqref{bir&dea} are related to the monic one as follows
\begin{equation} \label{qn&pin}
  \mathcal{Q}_n(x) = \left[(-1)^n\prod_{k=0}^{n-1}\la_k\right]^{-1} \pi_n(x),
\end{equation}
we have the following results for Berg-Letessier-Valent polynomials from the above lemma.
\begin{prop}
  Let $x=x_n:=(n+1/4)^4t$ with $t\in\mathbb{C}\cut[0,2^{10}]$, we have as $n\to\infty$,
 \begin{align*}
    p_n(x_n)&\sim2^{-8n-5/2}(n+\frac{1}{4})t^n(1-2^{10}/t)^{-1/4}\left[{1+\sqrt{1-2^{10}/t}\over 2}\right]^{2n+1/2}
  \nonumber\\&\times\exp\{{4(n+1/4)\over(2^{10}/t)^{1/4}}F(\arcsin (2^{10}/t)^{1/4},-1)\}.
  \end{align*}
  Here, the function $F$ is the elliptic integral of the first kind defined in \eqref{ellipticF}.
\end{prop}
\begin{proof}
  Recall the definition of $K_n$ in \eqref{BVkn-def}, we have
  $$K_n= \frac{\Gamma(\frac{n+\frac{3}{4}}{2}) \, \Gamma(\frac{n+1}{2})^2}{\Gamma(\frac{n+\frac{3}{2}}{2})^2 \, \Gamma(\frac{n+\frac{7}{4}}{2})}
  \sim \frac{2}{n+\frac{1}{4}},$$
  which yields $\mathcal{Q}_n(x) = (-1)^n K_n p_n(x) \sim(-1)^n\frac{2}{n+\frac{1}{4}}p_n(x). $
  Moreover, by applying Stirling's formula, we obtain
  \begin{align*}
    \prod_{k=0}^{n-1}\la_k&={2^{8n}\G(1/4+n)\G(1/2+n)^2\G(3/4+n)\over\G(1/4)\G(1/2)^2\G(3/4)}
    \\&\sim{2^{8n}\G(n)^4n^2\over\G(1/4)\G(1/2)^2\G(3/4)}
    \\&\sim{2^{8n}(2\pi)^2(n/e)^{4n}\over\G(1/4)\G(1/2)^2\G(3/4)}
    \\&\sim2^{8n}2\sqrt2(n/e)^{4n}.
  \end{align*}
  It is readily seen that
  $$\pi_n(x)=(-1)^n\prod_{k=0}^{n-1}\la_k\mathcal{Q}_n(x)\sim2^{8n+3/2}(n/e)^{4n}\frac{2}{n+\frac{1}{4}}p_n(x).$$
  Since $\pi_n(x)$ satisfies the asymptotic formula in (\ref{pi-outer-general}) with $\s=1/4$, $p=4$, $b=2^8$, $u=2$ and $v=0$, our result follows from the above formula.
\end{proof}
Again, using an argument similar to that in the proof of \cite[Theorem 3.2]{xswang2012}, we obtain the following result in the oscillatory region.
\begin{prop}
  Let $x=x_n:=(n+1/4)^4t$ with $t=2^{10}\cos^2\theta, \theta \in [\delta,\frac{\pi}{2}-\delta]$, we have as $n\to\infty$,
 \begin{align} \label{BV-pn-asy}
  p_n(x_n)&\sim \sqrt{2} (n+\frac{1}{4}) \frac{1}{(2^{10} - t)^{\frac{1}{4}}} \exp\left[ (n+\frac{1}{4}) t^{\frac{1}{4}} 2^{-5/2} \rho \right]
  \nonumber\\&
  \times\cos\{{\pi\over4}-(2n+1/2)\theta+ (n+\frac{1}{4}) \int_{t}^{2^{10}}
  \frac{(t/s)^{1/4}}{\sqrt{2^{10} s -s^2}} ds \}
\end{align}
where
\begin{equation*}
  \rho = \int_1^\infty \frac{du}{u^{\frac{1}{4}}\sqrt{u^2-u}} = \frac{\sqrt{\pi} \, \Gamma(\frac{1}{4})}{\Gamma(\frac{3}{4})}.
\end{equation*}
\end{prop}
Then we have a result in the interval containing $t_+=2^{10}$.
\begin{thm}
  Let $\nu = n+\frac{1}{4}$. With $K_n$ and $U(t)$ defined in \eqref{BVkn-def}, \eqref{BVzeta-def1} and \eqref{BVzeta-def2}, respectively, we have
  \begin{eqnarray} \label{BV-mainfor1}
    \mathcal{Q}_n(\nu^4 \, t) &\sim &  (-1)^n K_n \sqrt{2\pi} \; \nu t^{\frac{1}{4}} \exp\left( \frac{\sqrt{\pi} \, \Gamma(\frac{1}{4})}{ 2^{5/2}\,\Gamma(\frac{3}{4})} \nu t^{\frac{1}{4}} \right) \left(\frac{U(t)}{t(t-2^{10}) }\right)^{\frac{1}{4}} \nonumber \\
     && \times \left[ \Ai(\nu^{\frac{2}{3}} U(t)) \sum_{s=0}^\infty \frac{\tilde{A}_s(U)}{\nu^{s-\frac{1}{6}}} + \Ai'(\nu^{\frac{2}{3}} U(t)) \sum_{s=0}^\infty \frac{\tilde{B}_s(U)}{\nu^{s+\frac{1}{6}}} \right],
  \end{eqnarray}
  uniformly for $t$ in compact subsets of $(0,\infty)$.
\end{thm}

\begin{proof}

When $t<2^{10}$, recall the asymptotic expansions for Airy function in \eqref{Airy-asy1} and \eqref{Airy-asy2} again, we have from \eqref{BV-Pn} and \eqref{BV-Qn}
\begin{align*}
  p_n(x) &=  C_1(x) P_n(x) + C_2(x) Q_n(x) \sim C_1(x)  \frac{2^{9/2} }{[t (2^{10}-t)]^{\frac{1}{4}} }
   \frac{1}{\sqrt{\pi} }  \cos \biggl( \frac{2\nu}{3} (-U(t))^{\frac{3}{2}} -\frac{\pi}{4}\biggr) \nonumber \\
   & \quad - C_2(x)  \frac{2^{9/2} }{[t (2^{10}-t)]^{\frac{1}{4}} } \frac{1}{\sqrt{\pi} }  \sin \biggl( \frac{2\nu}{3} (-U(t))^{\frac{3}{2}} -\frac{\pi}{4} \biggr) \nonumber \\
   &= C_1(x)  \frac{2^{9/2} }{[t (2^{10}-t)]^{\frac{1}{4}} }
   \frac{1}{\sqrt{\pi} }  \cos \biggl( \nu( \cos^{-1} \left( \frac{t-2^9}{2^9} \right) - \int_t^{2^{10}} \frac{(t/s)^{\frac{1}{4}}}{\sqrt{s(2^{10}-s)}} ds)  -\frac{\pi}{4}\biggr) \nonumber \\
   &  \quad - C_2(x) \frac{2^{9/2} }{[t (2^{10}-t)]^{\frac{1}{4}} } \frac{1}{\sqrt{\pi} }   \sin \biggl( \nu( \cos^{-1} \left( \frac{t-2^9}{2^9} \right) - \int_t^{2^{10}} \frac{(t/s)^{\frac{1}{4}}}{\sqrt{s(2^{10}-s)}} ds)  -\frac{\pi}{4}\biggr). \nonumber
\end{align*}
Comparing the above formula and \eqref{BV-pn-asy}, we have
\begin{equation*}
  C_1(x) = \frac{\sqrt{\pi}\, x^{\frac{1}{4}}}{2^4} \exp\left( x^{\frac{1}{4}} \frac{\sqrt{\pi} \, \Gamma(\frac{1}{4})}{ 2^{5/2}\,\Gamma(\frac{3}{4})}  \right),  \qquad C_2(x) = 0.
\end{equation*}
This completes the proof of the theorem.
\end{proof}

\begin{rmk}
  Like in Remark \ref{fn-lim-measure}, we can obtain the following limiting zero distribution for $\mathcal{Q}_n(2^{10} n^4 x)$ from the recurrence relation \eqref{Berg&Val}
  \begin{equation*}
    \frac{1}{\pi} \int_{x^{1/4}}^1 \frac{1}{\sqrt{x(s^4-x)}} ds, \qquad x \in [0,1].
  \end{equation*}
\end{rmk}

Since the above theorem is uniformly valid in the neighbourhood of the large transition point $2^{10}$, we have the following asymptotic formula for $\mathcal{Q}_n(x)$.

\begin{cor} \label{BV-cor1}
  Let $\nu = n+\frac{1}{4}$ and $K_n$ be given in \eqref{BVkn-def}. Uniformly for a bounded real number $s$, we have, for $x = 2^{10} \nu^4 + 2^{10} \sqrt[3]{4} \,s  \,\nu^{\frac{10}{3}}$
  \begin{equation} \label{bv-airy}
    x^{-\frac{1}{4}} \exp\left( - \frac{\sqrt{\pi} \, \Gamma(\frac{1}{4})}{ 2^{5/2}\,\Gamma(\frac{3}{4})} x^{\frac{1}{4}} \right)  \mathcal{Q}_n(x) =  \frac{(-1)^n K_n \sqrt{\pi} \,  \nu^{\frac{1}{6}}}{16 \sqrt[3]{4} }  \left[ \Ai(s)  + O(\nu^{-\frac{2}{3}}) \right], \quad \textrm{as $\nu \to \infty.$}
  \end{equation}
\end{cor}
\begin{proof}
  Let $t = 2^{10} + 2^{10} \sqrt[3]{4} \,s \nu^{-\frac{2}{3}}$ in \eqref{BV-mainfor1} and recall the asymptotic formula of $U(t)$ in \eqref{BVUt-asy}, we have after some computations
  \begin{equation*}
    x^{-\frac{1}{4}} \exp\left( - \frac{\sqrt{\pi} \, \Gamma(\frac{1}{4})}{ 2^{5/2}\,\Gamma(\frac{3}{4})} x^{\frac{1}{4}} \right) \mathcal{Q}_n(x) =  (-1)^n K_n\sqrt{2\pi} \left( \frac{1}{32 \sqrt[6]{2}} \right)  \nu^{\frac{1}{6}} \left[ \Ai(s)  + O(\nu^{-\frac{2}{3}}) \right].
  \end{equation*}
  This above formula proves our corollary.
\end{proof}
\begin{rmk}
  We can put the left-hand side of \eqref{bv-airy} into its orthonormal form as in Remark \ref{fn-hat-rmk}.
  Let $\widehat {\mathcal{Q}}_{n}(x)$ be the orthonormal Berg-Letessier-Valent polynomials, then \eqref{bv-airy} can be rewritten as
  \begin{equation} \label{bv-hat-airy}
    x^{-\frac{1}{4}} \exp\left( - \frac{\sqrt{\pi} \, \Gamma(\frac{1}{4})}{ 2^{5/2}\,\Gamma(\frac{3}{4})} x^{\frac{1}{4}} \right)  \widehat {\mathcal{Q}}_n(x) =  \frac{ \nu^{-\frac{11}{6}}}{16 \sqrt[3]{4} }  \left[ \Ai(s)  + O(\nu^{-\frac{2}{3}}) \right], \quad \textrm{as $\nu \to \infty.$}
  \end{equation}
  Unlike \eqref{fn-hat-airy}, we don't have $w(x)$ on the left hand side of the above equation because the weight function $w(x)$ is unknown for the Berg-Letessier-Valent polynomials. But it seems reasonable to conjecture that at least one of the weight functions for the Berg-Letessier-Valent polynomials should behave like
  \begin{equation*}
    x^{-\frac{1}{2}} \exp\left( - \frac{\sqrt{\pi} \, \Gamma(\frac{1}{4})}{ 2^{3/2}\,\Gamma(\frac{3}{4})} x^{\frac{1}{4}} \right) \qquad \textrm{as }x \to \infty.
  \end{equation*}
  We shall revisit this issue in \S 6.
\end{rmk}

We have the following approximation.
\begin{cor}
  Let $x_{n,k}$ be the zeros of $\mathcal{Q}_n(x)$ such that $x_{n,1}>x_{n,2}> \cdots>x_{n,n}$, and $\mathfrak{a}_k$ be the zeros of the Airy function $\Ai(x)$ in the descending order. Then we have for fixed $k$ and large $n$
  \begin{equation}\label{xnk-cor32}
    x_{n,k} = 2^{10} \nu^4 + 2^{10} \sqrt[3]{4} \; \mathfrak{a}_k \nu^{10/3} + O(\nu^{8/3}),
  \end{equation}
  where $\nu = n + \frac{1}{4}$.
\end{cor}
\begin{proof}
  A combination of Corollary \ref{BV-cor1} and Lemma \ref{heth-lem} immediately gives us the result.
\end{proof}

We also have a result in the interval containing $t_-=0$.
\begin{thm}
  Let $\nu = n+\frac{1}{4}$. With $K_n$ and $U^*(t)$ defined in \eqref{BVkn-def}, \eqref{BVzeta2-def1} and \eqref{BVzeta2-def2}, respectively, we have
  \begin{align}
    \mathcal{Q}_n(\nu^4 \, t) &\sim  K_n \sqrt{\pi} \; \nu^{\frac{3}{2}} t^{\frac{1}{4}} \exp\left( \frac{\sqrt{\pi} \, \Gamma(\frac{1}{4})}{ 2^{5/2}\,\Gamma(\frac{3}{4})} \nu t^{\frac{1}{4}} \right) \left(\frac{U^*(t)}{t(2^{10}-t) }\right)^{\frac{1}{4}} \nonumber \\
     & \hspace{-60pt} \times \biggl\{ \sin\left( \frac{\sqrt{\pi} \, \Gamma(\frac{1}{4})}{ 2^{5/2}\,\Gamma(\frac{3}{4})} \nu t^{\frac{1}{4}} \right) \left[ J_{\frac{1}{2}}(\nu {U^*}^{\frac{1}{2}}(t)) \sum_{s=0}^\infty \frac{\tilde{A}_s^*(U^*)}{\nu^{s}} + J_{\frac{3}{2}}(\nu {U^*}^{\frac{1}{2}}(t)) \sum_{s=0}^\infty \frac{\tilde{B}_s^*(U^*)}{\nu^{s}} \right] \nonumber \\
      &\hspace{-65pt} - \cos\left( \frac{\sqrt{\pi} \, \Gamma(\frac{1}{4})}{ 2^{5/2}\,\Gamma(\frac{3}{4})} \nu t^{\frac{1}{4}} \right) \left[ Y_{\frac{1}{2}}(\nu {U^*}^{\frac{1}{2}}(t)) \sum_{s=0}^\infty \frac{\tilde{A}_s^*(U^*)}{\nu^{s}} + Y_{\frac{3}{2}}(\nu {U^*}^{\frac{1}{2}}(t)) \sum_{s=0}^\infty \frac{\tilde{B}_s^*(U^*)}{\nu^{s}} \right] \biggr\}, \label{BV-Bessel-formula}
  \end{align}
  uniformly for $-\infty< t \leq M < 2^{10}$.
\end{thm}
\begin{proof}
When $0<t\leq M$, we have from \eqref{BV-Pn2} and \eqref{BV-Qn2}
\begin{eqnarray*}
  p_n(x) &=&  (-1)^n C_1^*(x) P^*_n(x) + (-1)^n C^*_2(x) Q^*_n(x)  \nonumber \\
  & \sim & (-1)^n 2^{9/2}\nu^{\frac{1}{2}}\left(\frac{ U^*(t)}{t(2^{10}-t)}\right)^{\frac{1}{4}}  \biggl[  C_1^*(x) J_{\frac{1}{2}}(\nu {U^*}^{\frac{1}{2}}(t))  + C^*_2(x) W_{\frac{1}{2}}(\nu {U^*}^{\frac{1}{2}}(t)) \biggr]  \nonumber
\end{eqnarray*}
As $p_n(x)$ is real, we may choose $C_1^*(x) = \hat{C}_1(x) + i \hat{C}_2(x)$ and $C_2^*(x) = \hat{C}_2(x)$ such that
\begin{align}
   p_n(x) &\sim  (-1)^n 2^{9/2}\nu^{\frac{1}{2}}\left(\frac{ U^*(t)}{t(2^{10}-t)}\right)^{\frac{1}{4}}  \biggl[  \hat{C}_1(x)  J_{\frac{1}{2}}(\nu {U^*}^{\frac{1}{2}}(t))  + \hat{C}_2(x) Y_{\frac{1}{2}}(\nu {U^*}^{\frac{1}{2}}(t)) \biggr] \nonumber \\
   &\hspace*{-2em}= (-1)^n 2^{5} \left(\frac{ 1 }{t(2^{10}-t)}\right)^{\frac{1}{4}} \frac{1}{\sqrt{\pi}}  \biggl\{  \hat{C}_1(x)  \sin \biggl[ \nu \biggl(\int_0^t \frac{(t/s)^{\frac{1}{4}}}{\sqrt{s(2^{10} - s)}} ds-\cos^{-1} \left( \frac{2^9-t}{2^9} \right) \biggr)  \biggr] \nonumber  \\
   &   - \hat{C}_2(x) \cos \biggl[ \nu \biggl(\int_0^t \frac{(t/s)^{\frac{1}{4}}}{\sqrt{s(2^{10} - s)}} ds-\cos^{-1} \left( \frac{2^9-t}{2^9} \right) \biggr)  \biggr]\biggr\}. \label{BV-pn-proof}
\end{align}

Since
\begin{equation*}
  \int_{t}^{2^{10}} \frac{(t/s)^{1/4}}{\sqrt{2^{10} s -s^2}} ds =  2^{-5/2} \rho \, t^{\frac{1}{4}} - \int_0^t \frac{(t/s)^{1/4}}{\sqrt{2^{10} s -s^2}} ds
\end{equation*}
and $2\theta = \pi - \cos^{-1} \left( \frac{2^9-t}{2^9} \right)$
we rewrite \eqref{BV-pn-asy} as
\begin{align*}
  p_n(x)&\sim \sqrt{2} (n+\frac{1}{4}) \frac{1}{(2^{10} - t)^{\frac{1}{4}}} \exp\left[ (n+\frac{1}{4}) t^{\frac{1}{4}} 2^{-5/2} \rho \right]
  \nonumber\\
  &   \times\cos\{-{\pi\over4}+\nu \pi - \nu \cos^{-1} \left( \frac{2^9-t}{2^9} \right) - \nu \, t^{\frac{1}{4}} 2^{-5/2} \rho  + \nu \int_{0}^{t} \frac{(t/s)^{1/4}}{\sqrt{2^{10} s -s^2}} ds \} \nonumber \\
  &= \sqrt{2} \frac{x^{\frac{1}{4}} (-1)^n }{ (2^{10}t - t^2)^{\frac{1}{4}}} \exp\left[ x^{\frac{1}{4}} 2^{-5/2} \rho \right] \nonumber \\
   & \times \cos\left[ \nu  \biggl(\int_0^t \frac{(t/s)^{\frac{1}{4}}}{\sqrt{s(2^{10} - s)}} ds-\cos^{-1} \left( \frac{2^9-t}{2^9} \right) \biggr) -x^{\frac{1}{4}} 2^{-5/2} \rho \right].
\end{align*}
Comparing \eqref{BV-pn-proof} and the above formula, we have
\begin{eqnarray*}
  \hat{C}_1(x) = \frac{\sqrt{\pi}\, x^{\frac{1}{4}}}{2^{9/2}} \exp\left( x^{\frac{1}{4}} \frac{\sqrt{\pi} \, \Gamma(\frac{1}{4})}{ 2^{5/2}\,\Gamma(\frac{3}{4})}  \right)\sin\left( x^{\frac{1}{4}} \frac{\sqrt{\pi} \, \Gamma(\frac{1}{4})}{ 2^{5/2}\,\Gamma(\frac{3}{4})}  \right), \\
  \hat{C}_2(x) =- \frac{\sqrt{\pi}\, x^{\frac{1}{4}}}{2^{9/2}} \exp\left( x^{\frac{1}{4}} \frac{\sqrt{\pi} \, \Gamma(\frac{1}{4})}{ 2^{5/2}\,\Gamma(\frac{3}{4})}  \right)\cos\left( x^{\frac{1}{4}} \frac{\sqrt{\pi} \, \Gamma(\frac{1}{4})}{ 2^{5/2}\,\Gamma(\frac{3}{4})}  \right).
\end{eqnarray*}
This completes the proof of the theorem.
\end{proof}

\begin{rmk}
  To get the asymptotic formula for $t<0$ from \eqref{BV-Bessel-formula}, one needs to consider the value $t\pm i \varepsilon$ and take the limit as $\varepsilon \to 0.$ The interested readers may compare the formula \eqref{BV-Bessel-formula} with the Bessel-type asymptotic expansion for the Laguerre-type orthogonal polynomials in \cite[Eq. (2.15)]{van2007}.
\end{rmk}

\begin{rmk}
  From our expansion \eqref{BV-Bessel-formula}, it is also possible to study the smallest zeros of $\mathcal{Q}_n(x)$. Like in Corollary \ref{BV-cor1}, we have, as $\nu \to \infty$,
    \begin{eqnarray*}
    && x^{-\frac{1}{4}} \exp\left( - \frac{\sqrt{\pi} \, \Gamma(\frac{1}{4})}{ 2^{5/2}\,\Gamma(\frac{3}{4})} x^{\frac{1}{4}} \right)  \mathcal{Q}_n(x) \\
    && \qquad =  \frac{ K_n \sqrt{\nu\pi} }{16 \sqrt{2} }  \left[ \sin\left( \frac{\sqrt{\pi} \, \Gamma(\frac{1}{4})}{ 2^{5/2}\,\Gamma(\frac{3}{4})} x^{\frac{1}{4}} \right) J_{\frac{1}{2}}(s)- \cos\left( \frac{\sqrt{\pi} \, \Gamma(\frac{1}{4})}{ 2^{5/2}\,\Gamma(\frac{3}{4})} x^{\frac{1}{4}} \right) Y_{\frac{1}{2}}(s) + O(\nu^{-1}) \right],
  \end{eqnarray*}
  uniformly for a bounded real number $s$, where $x = 2^{8} \nu^{2} \, s$. Then approximations for the smallest zeros can be obtained from the above formula and Lemma 2.2. However, due to the sine and cosine terms, the formula is not as elegant as \eqref{xnk-cor32}. We leave it to interested readers.
  Similar situation happens for the smallest zeros of Conrad-Flajolet polynomials in the subsequent two sections.
\end{rmk}


\section{Conrad--Flajolet polynomials I} \label{cf1-sec}

We can easily get a bound for the largest zero for $\mathcal{Q}_n(x)$ from Theorem \ref{thmIsmLi}.

\begin{prop}
  Let $x_{n,k}$ be zeros of $\mathcal{Q}_n(x)$ such that $x_{n,1}>x_{n,2}> \cdots>x_{n,n}$, then we have the following bounds for all $n\geq 1$
  \begin{equation*}
    x_{n,1} < 108n^3 + 108 cn^2,
  \end{equation*}
  where $c$ is the positive constant given in \eqref{masterI}.
\end{prop}
\begin{proof}
  Recall the recurrence coefficients of $\mathcal{Q}_n(x)$ in \eqref{masterI} and choose $c_n =1/4$ in Theorem \ref{thmIsmLi}. Then the result follows.
\end{proof}
Note that we don't provide a bound for the smallest zero in this case.  The reason is that, although the chain sequence method can give us an estimation, it is not useful. In fact, we know that all the zeros of Conrad--Flajolet polynomials should be positive, but Theorem \ref{thmIsmLi} can only give us $ x_{n,n} > -6n-3. $ For the similar reason, we only consider the largest zero in Proposition \ref{M2-prop-cs}, too.

\subsection{Difference equation method}

As usually, we introduce $\mathcal{Q}_n(x) = (-1)^n K_n p_n(x)$ with
\begin{equation} \label{M1kn-def}
  K_n:= \frac{ \Gamma(\frac{n+\frac{c}{3}+1}{2})^2}{\Gamma(\frac{n+\frac{c}{3}+\frac{5}{3}}{2})^2}
\end{equation}
to arrive at the standard form. Note that
\begin{equation*}
  \frac{K_{n+1}}{K_{n-1}} = \frac{(3n+c)^2}{(3n+c+2)^2},
\end{equation*}
then the recurrence relation \eqref{bir&dea} with \eqref{masterI} reduces to the following standard form
\begin{equation} \label{M1-3term-new}
  p_{n+1}(x) - (A_n x + B_n) \, p_n(x) + p_{n-1}(x) = 0.
\end{equation}
As $n \to \infty$, the recurrence coefficients satisfy the following expansions
\begin{eqnarray}
  A_n = \frac{K_n}{\lambda_n \cdot K_{n+1}} \sim  n^{-\theta} \sum_{s=0}^\infty \frac{\alpha_s}{n^s}, \qquad
  B_n = -\frac{(\lambda_n + \mu_n)K_n}{\lambda_n \cdot K_{n+1}} \sim \sum_{s=0}^\infty \frac{\beta_s}{n^s} \label{M1ABn-asy}
\end{eqnarray}
with $\theta=3$,
\begin{eqnarray}
  \alpha_0 =  \frac{1}{27}, \qquad \quad \alpha_1 =  -\frac{c+1}{27} \label{M1alpha-s}
\end{eqnarray}
and
\begin{eqnarray}
  \beta_0 =  -2, \qquad \beta_1 = 0, \qquad \beta_2=\frac{2}{9}. \label{M1beta-s}
\end{eqnarray}
Let $x=(n+{c+1\over3})^3t$, then the roots of the characteristic equation \eqref{chara-eqn} coincide when
$\alpha_0 t_{\pm} + \beta_0 = \pm 2$, with $\alpha_0$ and $\beta_0$ given above. This gives us the following two transition points
\begin{equation*}
  t_+ = 108, \qquad t_- = 0.
\end{equation*}
These transition points are similar to what we have for Berg-Letessier-Valent polynomials. So the two Airy-type and Bessel-type expansions follow.

\begin{prop} \label{cf1-prop1}
  When $n$ is large, $p_n(x)$ in \eqref{M1-3term-new} can be expressed as
  \begin{equation*}
   p_n(x) = C_1(x) P_n(x) + C_2(x) Q_n(x),
  \end{equation*}
  where $C_1(x)$ and $C_2(x)$ are two $n$-independent functions, $P_n(x)$ and $Q_n(x)$ are two linearly independent solutions of \eqref{M1-3term-new} satisfying the following Airy-type asymptotic expansions in the neighbourhood of $t_+=108$
  \begin{equation} \label{M1-Pn}
  P_n(\nu^3 t) \sim  3\sqrt{6} \left(\frac{ U(t)}{t(t -108)}\right)^{\frac{1}{4}} \left[ \Ai(\nu^{\frac{2}{3}} U(t)) \sum_{s=0}^\infty \frac{\tilde{A}_s(U)}{\nu^{s-\frac{1}{6}}} + \Ai'(\nu^{\frac{2}{3}} U(t)) \sum_{s=0}^\infty \frac{\tilde{B}_s(U)}{\nu^{s+\frac{1}{6}}} \right]
\end{equation}
and
\begin{equation} \label{M1-Qn}
  Q_n(\nu^3 t) \sim 3\sqrt{6} \left(\frac{ U(t)}{t(t -108)}\right)^{\frac{1}{4}}  \left[ \Bi(\nu^{\frac{2}{3}} U(t)) \sum_{s=0}^\infty \frac{\tilde{A}_s(U)}{\nu^{s-\frac{1}{6}}} + \Bi'(\nu^{\frac{2}{3}} U(t)) \sum_{s=0}^\infty \frac{\tilde{B}_s(U)}{\nu^{s+\frac{1}{6}}} \right].
\end{equation}
Here $\nu = n + \frac{c+1}{3}$, the leading coefficients are given by
\begin{equation*}
  \tilde{A}_0(U) = 1, \qquad \tilde{B}_0(U) = 0
\end{equation*}
and $U(t)$ is defined as
\begin{eqnarray}
  \frac{2}{3} [U(t)]^{\frac{3}{2}} &= & \frac{\sqrt[3]{2t}}{6} \; \mathcal{B}_{1-\frac{108}{t}} (\frac{1}{2}, \frac{1}{3}) - \log \frac{t - 54 + \sqrt{t(t - 108)}}{54 }, \qquad t \geq 108, \label{M1zeta-def1} \\
  \frac{2}{3} [-U(t)]^{\frac{3}{2}} &=& \cos^{-1} \left( \frac{t-54}{54} \right) - \frac{\sqrt[3]{2t}}{6} \; \mathcal{B}_{1-\frac{t}{108}} (\frac{1}{2}, \frac{1}{6}) \qquad \qquad 0<t < 108; \label{M1zeta-def2}
\end{eqnarray}
see \eqref{icbeta} for the definition of $\mathcal{B}_x(a,b)$.
\end{prop}
\begin{proof}
  With
  \begin{equation*}
  \tau_0 = -\frac{\alpha_1t_+ + \beta_1}{(2-\beta_0) \theta} = \frac{c+1}{3} \qquad \textrm{and} \qquad \nu = n + \tau_0,
\end{equation*}
our results follow from the main Theorem in \cite{wang-wong2003}.
\end{proof}

\begin{prop}
  When $n$ is large, $p_n(x)$ in \eqref{M1-3term-new} can be expressed as
  \begin{equation*}
  p_n(x) = C_1^*(x)(-1)^n P_n^*(x) + C_2^*(x) (-1)^n Q_n^*(x)
  \end{equation*}
  where $C_1^*(x)$ and $C_2^*(x)$ are two $n$-independent functions, $P_n^*(x)$ and $Q_n^*(x)$ are two linearly independent solutions of \eqref{M1-3term-new} satisfying the following Bessel-type asymptotic expansions in the neighbourhood of $t_-=0$
  \begin{equation} \label{M1-Pn2}
  P_n^*(\nu^3 t) \sim 3\sqrt{6}\; \nu^{\frac{1}{2}} \left(\frac{ U^*(t)}{t(108-t)}\right)^{\frac{1}{4}} \left[ J_{\frac{1}{3}}(\nu {U^*}^{\frac{1}{2}}(t)) \sum_{s=0}^\infty \frac{\tilde{A}_s^*(U^*)}{\nu^{s}} + J_{\frac{4}{3}}(\nu {U^*}^{\frac{1}{2}}(t)) \sum_{s=0}^\infty \frac{\tilde{B}_s^*(U^*)}{\nu^{s}} \right]
\end{equation}
and
\begin{equation} \label{M1-Qn2}
  Q_n^*(\nu^3 t) \sim 3\sqrt{6}\; \nu^{\frac{1}{2}} \left(\frac{ U^*(t)}{t(108-t)}\right)^{\frac{1}{4}} \left[ W_{\frac{1}{3}}(\nu {U^*}^{\frac{1}{2}}(t)) \sum_{s=0}^\infty \frac{\tilde{A}_s^*(U^*)}{\nu^{s}} + W_{\frac{4}{3}}(\nu {U^*}^{\frac{1}{2}}(t)) \sum_{s=0}^\infty \frac{\tilde{B}_s^*(U^*)}{\nu^{s}} \right].
\end{equation}
Here $\nu = n + \frac{c+1}{3}$, $W_\alpha(x)$ is defined in \eqref{w-alpha-def}, the leading coefficients are given by
\begin{equation}\label{M1-AB}
  \tilde{A}_0^*(U^*) = 1, \qquad \tilde{B}_0^*(U^*) = 0
\end{equation}
and $U^*(t)$ is defined as
\begin{eqnarray}
  [-U^*(t)]^{\frac{1}{2}} &= & \int_{t}^0 \frac{(t/s)^{\frac{1}{3}}}{\sqrt{s(s - 108)}} ds -\log \frac{54-t + \sqrt{t(t - 108)}}{54 } , \quad t \leq 0 \label{M1zeta2-def1} \\
  \ [U^*(t)]^{\frac{1}{2}} &=& \frac{\sqrt[3]{2t}}{6} \; \mathcal{B}_{\frac{t}{108}} (\frac{1}{6}, \frac{1}{2}) -\cos^{-1}\left( \frac{54-t}{54} \right),  \qquad \qquad 0<t < 108; \label{M1zeta2-def2}
\end{eqnarray}
see \eqref{icbeta} for the definition of $\mathcal{B}_x(a,b)$.
\end{prop}
\begin{proof}
  With the asymptotic expansions for $A_n$ and $B_n$ in \eqref{M1ABn-asy}, our results follow from the main Theorem in \cite{Cao:Li}.
\end{proof}
\begin{rmk}
  Again, the integral in \eqref{M1zeta2-def1} can be written as a hypergeometric function.
  \begin{equation*}
    \int_{t}^0 \frac{(t/s)^{\frac{1}{3}}}{\sqrt{s(s - 108)}} ds = \sqrt[3]{2} (-t)^{\frac{1}{2}} \; {}_2F_1(\frac{1}{6},\frac{1}{2},\frac{7}{6},t).
  \end{equation*}
  $U(t)$ and  $U^*(t)$ are two monotonically increasing function in the neighborhood of 108 and 0, respectively,  with the following asymptotic formulas
  \begin{eqnarray} \label{M1Ut-asy}
    U(t)=\frac{t-108}{54 \sqrt[3]{18}} + O(t-108)^2, \qquad \textrm{as } t \to 108.
  \end{eqnarray}
  and
  \begin{equation}\label{M1-U*}
    U^*(t) = \frac{4t}{27} + O(t^2) \qquad \textrm{as } t \to 0.
  \end{equation}
\end{rmk}


\subsection{Determination of $C_1(x)$ and $C_2(x)$}

According to Lemma \ref{lemma-birth}, we have the following result.
\begin{prop}
  Let $x=x_n:=(n+(c+1)/3)^3t$ with $t\in\mathbb{C}\cut[0,108]$, we have as $n\to\infty$,
  \begin{align*}
    p_n(x_n)&\sim {\G((c+1)/3)\G((c+2)/3)^2t^n\over3^{3n}n^{c+1}(2\pi/n)^{3/2} 2^{2/3}}(1-108/t)^{-1/4}\left[{1+\sqrt{1-108/t}\over 2}\right]^{2n+1+{(2c-1)\over 3}}
  \nonumber\\&\times\exp\{\int_0^1{3n+c+1\over\sqrt{1-108s^3/t}}ds\}.
  \end{align*}
\end{prop}
\begin{proof}
 Recall the rates given in \eqref{masterI} and $K_n$ defined in \eqref{M1kn-def}. By Stirling's formula, we have $K_n\sim 2^{2/3} n^{-2/3}$ and
  \begin{align*}
    \prod_{k=0}^{n-1}\la_k&={3^{3n}\G((c+1)/3+n)\G((c+2)/3+n)^2\over\G((c+1)/3)\G((c+2)/3)^2}
    \\&\sim{3^{3n}n^{c+5/3}\G(n)^3\over\G((c+1)/3)\G((c+2)/3)^2}
    \\&\sim{3^{3n}n^{c+5/3}(2\pi/n)^{3/2}(n/e)^{3n}\over\G((c+1)/3)\G((c+2)/3)^2}.
  \end{align*}
  It thus follows from \eqref{qn&pin} that
  $$\pi_n(x)=K_n\prod_{k=0}^{n-1}\la_kp_n(x)\sim{3^{3n}n^{c+1}(2\pi/n)^{3/2}(n/e)^{3n} 2^{2/3}\over\G((c+1)/3)\G((c+2)/3)^2}p_n(x).$$
  Since $\pi_n(x)$ satisfies the asymptotic formula \eqref{pi-outer-general} with $p=3$, $b=27$, $u=c+5/3$ and $v=c+1/3$, our result follows from the above formula.
\end{proof}
Again, using an argument similar to that in the proof of \cite[Theorem 3.2]{xswang2012}, we obtain the following result in the oscillatory region.
\begin{prop}
  Let $x=x_n:=(n+(c+1)/3)^3t$ with $t=108\cos^2\theta,\theta \in [\delta,\frac{\pi}{2}-\delta]$, we have as $n\to\infty$,
\begin{align} \label{M1-pn-asy}
    p_n(x_n)&\sim {3^{c+1} \G({c+1 \over 3})\G^2({c+2 \over 3}) \over 2^{7/6} \pi^{3/2} (108t-t^2)^{1/4}} \frac{1}{(n^3t)^{\frac{2c-1}{6}}} \exp\left( (n+{c+1\over 3}) (\frac{t}{108})^{1/3} \rho \right)
  \nonumber\\&
  \times \cos\{{\pi\over4}-(2n+{(2c+2)\over 3})\theta + (n+{ c+1\over 3 }) \int_{t}^{108}
  \frac{(t/s)^{1/3}}{\sqrt{108s-s^2}}ds \},
\end{align}
where
\begin{equation} \label{rho-M1}
  \rho = \int_1^\infty \frac{du}{u^{\frac{1}{3}}\sqrt{u^2-u}} = \frac{\sqrt{\pi} \, \Gamma(\frac{1}{3})}{\Gamma(\frac{5}{6})}.
\end{equation}
\end{prop}
Then we get the Airy-type expansion in the interval containing $t_+= 108.$
\begin{thm}
  Let $\nu = n+\frac{c+1}{3}$. With $K_n$ and $U(t)$ defined in \eqref{M1kn-def}, \eqref{M1zeta-def1} and \eqref{M1zeta-def2}, respectively, we have
  \begin{eqnarray} \label{M1-mainfor1}
    \mathcal{Q}_n(\nu^3 \, t) &\sim &  (-1)^n K_n {3^{c+1} \G({c+1 \over 3})\G^2({c+2 \over 3}) \over 2^{\frac{7}{6}} \pi \, \nu^{\frac{2c-1}{2}} \, t^{\frac{2c-1}{6}}}  \exp\left(  \frac{\sqrt{\pi} \, \Gamma(\frac{1}{3})}{3\sqrt[3]{4} \, \Gamma(\frac{5}{6})} \nu t^{\frac{1}{3}} \right) \left(\frac{U(t)}{t(t-108) }\right)^{\frac{1}{4}} \nonumber \\
     && \times \left[ \Ai(\nu^{\frac{2}{3}} U(t)) \sum_{s=0}^\infty \frac{\tilde{A}_s(U)}{\nu^{s-\frac{1}{6}}} + \Ai'(\nu^{\frac{2}{3}} U(t)) \sum_{s=0}^\infty \frac{\tilde{B}_s(U)}{\nu^{s+\frac{1}{6}}} \right],
  \end{eqnarray}
  uniformly for $0<\delta \leq t < \infty$.
\end{thm}

\begin{proof}

When $t<108$, recall the asymptotic expansions for Airy function in \eqref{Airy-asy1} and \eqref{Airy-asy2} again, we have from \eqref{M1-Pn} and \eqref{M1-Qn}
\begin{align*}
  p_n(x) &=  C_1(x) P_n(x) + C_2(x) Q_n(x) \sim C_1(x)  \frac{3\sqrt{6} }{[t (108-t)]^{\frac{1}{4}} }
   \frac{1}{\sqrt{\pi} }  \cos \biggl( \frac{2\nu}{3} (-U(t))^{\frac{3}{2}} -\frac{\pi}{4}\biggr) \nonumber \\
   & \quad - C_2(x)  \frac{3\sqrt{6}}{[t (108-t)]^{\frac{1}{4}} } \frac{1}{\sqrt{\pi} }  \sin \biggl( \frac{2\nu}{3} (-U(t))^{\frac{3}{2}} -\frac{\pi}{4} \biggr) \nonumber \\
   &= C_1(x)  \frac{3\sqrt{6} }{[t (108-t)]^{\frac{1}{4}} }
   \frac{1}{\sqrt{\pi} }  \cos \biggl( \nu( \cos^{-1} \left( \frac{t-54}{54} \right) - \int_t^{108} \frac{(t/s)^{\frac{1}{3}}}{\sqrt{s(108-s)}} ds)  -\frac{\pi}{4}\biggr) \nonumber \\
   &  \quad - C_2(x) \frac{3 \sqrt{6} }{[t (108-t)]^{\frac{1}{4}} } \frac{1}{\sqrt{\pi} }   \sin \biggl( \nu( \cos^{-1} \left( \frac{t-54}{54} \right) - \int_t^{108} \frac{(t/s)^{\frac{1}{3}}}{\sqrt{s(108-s)}} ds)  -\frac{\pi}{4}\biggr). \nonumber
\end{align*}
Comparing the above formula and \eqref{M1-pn-asy}, we have
\begin{equation*}
  C_1(x) = {3^{c-1/2} \G((c+1)/3)\G^2((c+2)/3) \over 2^{4/3} \pi x^{\frac{2c-1}{6}}}  \exp\left( (\frac{x}{108})^{\frac{1}{3}} \frac{\sqrt{\pi} \, \Gamma(\frac{1}{3})}{\Gamma(\frac{5}{6})}  \right),  \qquad C_2(x) = 0.
\end{equation*}
This completes the proof of the theorem.
\end{proof}

\begin{rmk}
  Like in Remark \ref{fn-lim-measure}, we can obtain the following limiting zero distribution for $\mathcal{Q}_n(108 n^3 x)$ from the recurrence relation \eqref{masterI}
  \begin{equation} \label{masterI-measure}
    \frac{1}{\pi} \int_{x^{1/3}}^1 \frac{1}{\sqrt{x(s^3-x)}} ds, \qquad x \in [0,1].
  \end{equation}
\end{rmk}

Since the above theorem is uniformly valid in the neighbourhood of the large transition point $108$, we have the following asymptotic formula for $\mathcal{Q}_n(x)$.

\begin{cor} \label{M1-cor1}
  Let $\nu = n+\frac{c+1}{3}$ and $K_n$ be given in \eqref{M1kn-def}. Uniformly for a bounded real number $s$, we have, for $x = 108 \nu^3 + 54 \sqrt[3]{18} \,s  \,\nu^{\frac{7}{3}}$
  \begin{eqnarray}
    &&\hspace*{-1cm} x^{\frac{2c-1}{6}} \exp\left( - \frac{\sqrt{\pi} \, \Gamma(\frac{1}{3})}{3\sqrt[3]{4} \, \Gamma(\frac{5}{6})} x^{\frac{1}{3}} \right) \mathcal{Q}_n(x) =  (-1)^n K_n {3^{c-\frac{2}{3} } \G({c+1 \over 3})\G^2({c+2 \over 3}) \nu^{\frac{1}{6}} \over 4 \pi }    \left[ \Ai(s)  + O(\nu^{-\frac{2}{3}}) \right], \nonumber \\
     && \hspace{10cm} \textrm{as $\nu \to \infty.$} \label{M1-airy}
  \end{eqnarray}
\end{cor}
\begin{proof}
   Let $t = 108 + 54 \sqrt[3]{18} \,s \nu^{-\frac{2}{3}}$ in \eqref{M1-mainfor1} and recall the asymptotic formula for $U(t)$ in \eqref{M1Ut-asy}, we have after some computations
  \begin{eqnarray*}
    &&x^{\frac{2c-1}{6}} \exp\left( - \frac{\sqrt{\pi} \, \Gamma(\frac{1}{3})}{3\sqrt[3]{4} \, \Gamma(\frac{5}{6})} x^{\frac{1}{3}} \right) \mathcal{Q}_n(x) =   (-1)^n K_n {3^{c+1} \G({c+1 \over 3})\G^2({c+2 \over 3}) \over 2^{\frac{7}{6}} \pi }   \left( \frac{1}{3^{\frac{5}{3}} 2^{\frac{5}{6}} } + O(\nu^{-\frac{2}{3}}) \right) \nonumber \\
    &  & \qquad \times  \nu^{\frac{1}{6}} \left[ \Ai(s)  + O(\nu^{-\frac{2}{3}}) \right].
  \end{eqnarray*}
  This above formula proves our corollary.
\end{proof}
\begin{rmk}
  Again we can put the left hand side of \eqref{M1-airy} into its orthonormal form.
  Let $\widehat {\mathcal{Q}}_{n}(x)$ be the orthonormal Conrad--Flajolet polynomials I, then \eqref{M1-airy} can be rewritten as
  \begin{eqnarray}
    && x^{\frac{2c-1}{6}} \exp\left( - \frac{\sqrt{\pi} \, \Gamma(\frac{1}{3})}{3\sqrt[3]{4} \, \Gamma(\frac{5}{6})} x^{\frac{1}{3}} \right) \widehat {\mathcal{Q}}_n(x) =   { \sqrt{c+1} \; \G(c+1 ) \nu^{-\frac{4}{3}} \over 2^{\frac{1}{3}}   \; 3^{\frac{5}{3}}  }    \left[ \Ai(s)  + O(\nu^{-\frac{2}{3}}) \right], \nonumber \\
     && \hspace{11cm} \textrm{as $\nu \to \infty.$} \label{M1-hat-airy}
  \end{eqnarray}
  According to the above formula, again it seems reasonable to conjecture that at least one of the weight functions for the Conrad--Flajolet polynomials I should behave like
  \begin{equation*}
    x^{\frac{2c-1}{3}} \exp\left( - \frac{\sqrt{\pi} \sqrt[3]{2} \, \Gamma(\frac{1}{3})}{3 \, \Gamma(\frac{5}{6})} x^{\frac{1}{3}} \right) \qquad \textrm{as }x \to \infty.
  \end{equation*}
\end{rmk}
We have the following approximation.
\begin{cor} \label{M1-corol}
  Let $x_{n,k}$ be the zeros of $\mathcal{Q}_n(x)$ such that $x_{n,1}>x_{n,2}> \cdots>x_{n,n}$, and $\mathfrak{a}_k$ be the zeros of the Airy function $\Ai(x)$ in the descending order. Then we have for fixed $k$ and large $n$
  \begin{equation*}
    x_{n,k} = 108 \nu^3 + 54 \sqrt[3]{18}\; \mathfrak{a}_k \nu^{7/3} + O(\nu^{5/3}),
  \end{equation*}
  where $\nu = n + \frac{c+1}{3}$.
\end{cor}
\begin{proof}
  A combination of Corollary \ref{M1-cor1} and Lemma \ref{heth-lem} gives us the result.
\end{proof}

We also get the Bessel-type expansion in the interval containing $t_-= 0.$

\begin{thm} \label{M1-thm2}
  Let $\nu = n+\frac{c+1}{3}$. With $K_n$ and $U^*(t)$ defined in \eqref{M1kn-def}, \eqref{M1zeta2-def1} and \eqref{M1zeta2-def2}, respectively, we have
  \begin{eqnarray*}
    \mathcal{Q}_n(\nu^3 \, t) &\sim&  K_n {3^{c+1} \G({c+1 \over 3})\G^2({c+2 \over 3}) \; \nu^{1-c} \over 2^{5/3} \pi \; t^{\frac{2c-1}{6}} }   \exp\left( \frac{\sqrt{\pi} \, \Gamma(\frac{1}{3})}{3\sqrt[3]{4} \, \Gamma(\frac{5}{6})} \nu \, t^{\frac{1}{3}} \right)  \left(\frac{U^*(t)}{t(108-t) }\right)^{\frac{1}{4}} \nonumber \\
     && \hspace{-60pt} \times \biggl\{ \sin\left( \frac{\sqrt{\pi} \, \Gamma(\frac{1}{3})}{2^{\frac{2}{3}} \sqrt{3} \,\Gamma(\frac{5}{6})} \, x^{\frac{1}{3}} -  \frac{c}{3} \pi \right) \left[ J_{\frac{1}{3}}(\nu {U^*}^{\frac{1}{2}}(t)) \sum_{s=0}^\infty \frac{\tilde{A}_s^*(U^*)}{\nu^{s}} + J_{\frac{4}{3}}(\nu {U^*}^{\frac{1}{2}}(t)) \sum_{s=0}^\infty \frac{\tilde{B}_s^*(U^*)}{\nu^{s}} \right] \nonumber \\
      &&\hspace{-65pt} - \cos\left( \frac{\sqrt{\pi} \, \Gamma(\frac{1}{3})}{2^{\frac{2}{3}} \sqrt{3} \,\Gamma(\frac{5}{6})} \, x^{\frac{1}{3}} -  \frac{c}{3} \pi \right) \left[ Y_{\frac{1}{3}}(\nu {U^*}^{\frac{1}{2}}(t)) \sum_{s=0}^\infty \frac{\tilde{A}_s^*(U^*)}{\nu^{s}} + Y_{\frac{4}{3}}(\nu {U^*}^{\frac{1}{2}}(t)) \sum_{s=0}^\infty \frac{\tilde{B}_s^*(U^*)}{\nu^{s}} \right] \biggr\},
      \nonumber
  \end{eqnarray*}
  uniformly for $-\infty< t \leq M < 108$.
\end{thm}
\begin{proof}
When $0<t\leq M$, we have from \eqref{M1-Pn2} and \eqref{M1-Qn2}
\begin{eqnarray*}
  p_n(x) &=&  (-1)^n C_1^*(x) P^*_n(x) + (-1)^n C^*_2(x) Q^*_n(x)  \nonumber \\
  & \sim & (-1)^n 3 \sqrt{6} \, \nu^{\frac{1}{2}}\left(\frac{ U^*(t)}{t(108-t)}\right)^{\frac{1}{4}}  \biggl[  C_1^*(x) J_{\frac{1}{3}}(\nu {U^*}^{\frac{1}{2}}(t))  + C^*_2(x) W_{\frac{1}{3}}(\nu {U^*}^{\frac{1}{2}}(t)) \biggr]  \nonumber
\end{eqnarray*}
Note the following asymptotic expansions for Bessel functions, as $x \to \infty$
\begin{eqnarray}
  J_\nu(x) \sim \left(\frac{2}{ \pi x}\right)^{\frac{1}{2}}  \cos ( x- \frac{\nu \pi}{2} -\frac{\pi}{4}), \qquad
  Y_\nu(x) \sim \left(\frac{2}{ \pi x}\right)^{\frac{1}{2}}  \sin ( x- \frac{\nu \pi}{2} -\frac{\pi}{4}). \label{Bessl-asy}
\end{eqnarray}
As $p_n(x)$ is real, we choose $C_1^*(x) = \hat{C}_1(x) + i \hat{C}_2(x)$ and $C_2^*(x) = \hat{C}_2(x)$ such that
\begin{align}
   p_n(x) &\sim   (-1)^n 3 \sqrt{6} \, \nu^{\frac{1}{2}}\left(\frac{ U^*(t)}{t(108-t)}\right)^{\frac{1}{4}}  \biggl[  \hat{C}_1(x) J_{\frac{1}{3}}(\nu {U^*}^{\frac{1}{2}}(t))  + \hat{C}_2(x) Y_{\frac{1}{3}}(\nu {U^*}^{\frac{1}{2}}(t)) \biggr] \nonumber \\
   &\sim (-1)^n \frac{6\sqrt{3}}{\sqrt{\pi}} \left(\frac{ 1 }{t(108-t)}\right)^{\frac{1}{4}}
   \nonumber\\
   &~\times\biggl\{  \hat{C}_1(x)  \cos \biggl[ \nu \biggl(\int_0^t \frac{(t/s)^{\frac{1}{3}}}{\sqrt{s(108 - s)}} ds-\cos^{-1} \left( \frac{54-t}{54} \right) \biggr) -\frac{5}{12} \pi  \biggr] \nonumber  \\
   & \qquad  + \hat{C}_2(x) \sin \biggl[ \nu \biggl(\int_0^t \frac{(t/s)^{\frac{1}{3}}}{\sqrt{s(108 - s)}} ds-\cos^{-1} \left( \frac{54-t}{54} \right) \biggr)  -\frac{5}{12} \pi \biggr]\biggr\}. \label{M1-pn-proof}
\end{align}

Since
\begin{equation*}
  \int_{t}^{108} \frac{(t/s)^{1/3}}{\sqrt{108 s -s^2}} ds =   \frac{\sqrt{\pi} \, \Gamma(\frac{1}{3})}{2^{\frac{2}{3}} \sqrt{3} \,\Gamma(\frac{5}{6})} \, t^{\frac{1}{3}} - \int_0^t \frac{(t/s)^{1/3}}{\sqrt{108 s -s^2}} ds
\end{equation*}
and $2\theta = \pi - \cos^{-1} \left( \frac{54-t}{54} \right),$
we rewrite \eqref{M1-pn-asy} as
\begin{align*}
  p_n(x)&\sim {3^{c+1} \G({c+1 \over 3})\G^2({c+2 \over 3}) \over 2^{7/6} \pi^{3/2} (108t-t^2)^{1/4}} \frac{1}{x^{\frac{2c-1}{6}}} \exp\left( (\frac{x}{108})^{1/3} \rho \right)
  \nonumber\\&~
  \times \cos\{-{\pi\over4}+\nu \pi - \nu \cos^{-1} \left( \frac{54-t}{54} \right) - \nu \frac{\sqrt{\pi} \, \Gamma(\frac{1}{3})}{2^{\frac{2}{3}} \sqrt{3} \,\Gamma(\frac{5}{6})} \, t^{\frac{1}{3}} + \nu \int_{0}^{t} \frac{(t/s)^{1/3}}{\sqrt{108s-s^2}}ds \} \nonumber \\
  &= {3^{c+1} \G({c+1 \over 3})\G^2({c+2 \over 3}) \over 2^{7/6} \pi^{3/2} } \frac{ (-1)^n }{x^{\frac{2c-1}{6}} (108t - t^2)^{\frac{1}{4}}} \exp\left( (\frac{x}{108})^{1/3} \rho \right) \nonumber \\
   &\hspace*{-2em} \times \cos\left[  \nu \biggl(\int_0^t \frac{(t/s)^{\frac{1}{3}}}{\sqrt{s(108 - s)}} ds-\cos^{-1} \left( \frac{54-t}{54} \right) \biggr)  -\frac{5}{12} \pi  - \frac{\sqrt{\pi} \, \Gamma(\frac{1}{3})}{2^{\frac{2}{3}} \sqrt{3} \,\Gamma(\frac{5}{6})} \, x^{\frac{1}{3}}+ (\frac{c}{3} + \frac{1}{2}) \pi \right].
\end{align*}
Comparing \eqref{M1-pn-proof} and the above formula, we have
\begin{eqnarray*}
  \hat{C}_1(x) = {3^{c-1/2} \G({c+1 \over 3})\G^2({c+2 \over 3}) \over 2^{13/6} \pi \, x^{\frac{2c-1}{6}} }   \exp\left( x^{\frac{1}{3}}  \frac{\sqrt{\pi} \, \Gamma(\frac{1}{3})}{3\sqrt[3]{4} \, \Gamma(\frac{5}{6})} \right) \sin\left( \frac{\sqrt{\pi} \, \Gamma(\frac{1}{3})}{2^{\frac{2}{3}} \sqrt{3} \,\Gamma(\frac{5}{6})} \, x^{\frac{1}{3}} -  \frac{c}{3} \pi \right), \\
   \hat{C}_2(x) = - {3^{c-1/2} \G({c+1 \over 3})\G^2({c+2 \over 3}) \over 2^{13/6} \pi \, x^{\frac{2c-1}{6}} }   \exp\left( x^{\frac{1}{3}}  \frac{\sqrt{\pi} \, \Gamma(\frac{1}{3})}{3\sqrt[3]{4} \, \Gamma(\frac{5}{6})} \right) \cos\left( \frac{\sqrt{\pi} \, \Gamma(\frac{1}{3})}{2^{\frac{2}{3}} \sqrt{3} \,\Gamma(\frac{5}{6})} \, x^{\frac{1}{3}} -  \frac{c}{3} \pi \right).
\end{eqnarray*}
This completes the proof of the theorem.
\end{proof}


\section{Conrad--Flajolet polynomials II}

Again we use Theorem \ref{thmIsmLi} to get a bound for the largest zero for $\mathcal{Q}_n(x)$ as follows.

\begin{prop} \label{M2-prop-cs}
  Let $x_{n,k}$ be zeros of $\mathcal{Q}_n(x)$ such that $x_{n,1}>x_{n,2}> \cdots>x_{n,n}$, then we have the following bounds for all $n\geq 1$
  \begin{equation*}
    x_{n,1} < 108n^3 + 108cn^2.
  \end{equation*}
\end{prop}
\begin{proof}
  Recall the recurrence coefficients of $\mathcal{Q}_n(x)$ in \eqref{masterII} and choose $c_n =1/4$ in Theorem \ref{thmIsmLi}. Then the result follows.
\end{proof}

\subsection{Difference equation method}

As what we have done before, we introduce $\mathcal{Q}_n(x) = (-1)^n K_n p_n(x)$ with
\begin{equation} \label{M2kn-def}
  K_n:= \frac{\Gamma(\frac{n+\frac{c}{3}+\frac{2}{3}}{2}) \, \Gamma(\frac{n+\frac{c}{3}+1}{2})^2}{\Gamma(\frac{n+\frac{c}{3}+\frac{4}{3}}{2})^2 \, \Gamma(\frac{n+\frac{c}{3}+\frac{5}{3}}{2})}
\end{equation}
to arrive at the standard form. Note that
\begin{equation*}
  \frac{K_{n+1}}{K_{n-1}} = \frac{(3n+c-1)(3n+c)^2}{(3n+c+1)^2(3n+c+2)},
\end{equation*}
then the recurrence relation \eqref{bir&dea} with \eqref{masterII} reduces to the standard form \eqref{M1-3term-new}.
As $n \to \infty$, the recurrence coefficients satisfy the following expansions
\begin{eqnarray}
  A_n = \frac{K_n}{\lambda_n \cdot K_{n+1}}  \sim  n^{-\theta} \sum_{s=0}^\infty \frac{\alpha_s}{n^s}, \qquad
  B_n = -\frac{(\lambda_n + \mu_n)K_n}{\lambda_n \cdot K_{n+1}} \sim \sum_{s=0}^\infty \frac{\beta_s}{n^s} \label{M2ABn-asy}
\end{eqnarray}
with $\theta=3$,
\begin{eqnarray}
  \alpha_0 =  \frac{1}{27}, \qquad \quad \alpha_1 =  -\frac{2c+1}{54} \label{M2alpha-s}
\end{eqnarray}
and
\begin{eqnarray}
  \beta_0 =  -2, \qquad \beta_1 = 0, \qquad \beta_2 = \frac{5}{36}. \label{M2beta-s}
\end{eqnarray}
Since $\alpha_0$ and $\beta_0$ are the same as those in \eqref{M1alpha-s} and \eqref{M1beta-s}, we get the same transition points as follows
\begin{equation*}
  t_+ = 108, \qquad t_- = 0.
\end{equation*}
Again we get two different types of asymptotic expansions near $t_+$ and $t_-$. Due to the similarity between the current case and that in Sec. \ref{cf1-sec}, we get the same Airy-type expansion as that in Proposition \ref{cf1-prop1} except that in this case, $\nu= n+\frac{2c+1}{6}$.

Because the higher order coefficients in \eqref{M2alpha-s} and \eqref{M2beta-s} are different from those in \eqref{M1alpha-s} and \eqref{M1beta-s}, we have a corresponding difference in the order of Bessel functions. The Bessel-type expansion near $t_-$ is given as follows.

\begin{prop}
  When $n$ is large, $p_n(x)$ in \eqref{M1-3term-new} can be expressed as
  \begin{equation*}
  p_n(x) = C_1^*(x)(-1)^n P_n^*(x) + C_2^*(x) (-1)^n Q_n^*(x)
  \end{equation*}
  where $C_1^*(x)$ and $C_2^*(x)$ are two $n$-independent functions, $P_n^*(x)$ and $Q_n^*(x)$ are two linearly independent solutions of \eqref{M1-3term-new} satisfying the following Bessel-type asymptotic expansions in the neighbourhood of $t_-=0$
  \begin{equation*}
  P_n^*(\nu^3 t) \sim 3\sqrt{6}\; \nu^{\frac{1}{2}} \left(\frac{ U^*(t)}{t(108-t)}\right)^{\frac{1}{4}} \left[ J_{\frac{2}{3}}(\nu {U^*}^{\frac{1}{2}}(t)) \sum_{s=0}^\infty \frac{\tilde{A}_s^*(U^*)}{\nu^{s}} + J_{\frac{5}{3}}(\nu {U^*}^{\frac{1}{2}}(t)) \sum_{s=0}^\infty \frac{\tilde{B}_s^*(U^*)}{\nu^{s}} \right]
\end{equation*}
and
\begin{equation*}
  Q_n^*(\nu^3 t) \sim 3\sqrt{6}\; \nu^{\frac{1}{2}} \left(\frac{ U^*(t)}{t(108-t)}\right)^{\frac{1}{4}} \left[ W_{\frac{2}{3}}(\nu {U^*}^{\frac{1}{2}}(t)) \sum_{s=0}^\infty \frac{\tilde{A}_s^*(U^*)}{\nu^{s}} + W_{\frac{5}{3}}(\nu {U^*}^{\frac{1}{2}}(t)) \sum_{s=0}^\infty \frac{\tilde{B}_s^*(U^*)}{\nu^{s}} \right].
\end{equation*}
Here $\nu = n + \frac{2c+1}{6}$, $W_\alpha(x)$ is defined in \eqref{w-alpha-def}, the leading coefficients are given by
\begin{equation*}
  \tilde{A}_0^*(U^*) = 1, \qquad \tilde{B}_0^*(U^*) = 0.
\end{equation*}
and $U^*(t)$ is the same as defined in \eqref{M1zeta2-def1} and \eqref{M1zeta2-def2}.
\end{prop}
\begin{proof}
  With the asymptotic expansions for $A_n$ and $B_n$ in \eqref{M2ABn-asy}, our results follow from the main Theorem in \cite{Cao:Li}.
\end{proof}


\subsection{Determination of $C_1(x)$ and $C_2(x)$}

According to Lemma \ref{lemma-birth}, we have the following result.
\begin{prop}
  Let $x=x_n:=(n+(2c+1)/6)^3t$ with $t\in\mathbb{C}\cut[0,108]$, we have as $n\to\infty$,
  \begin{align*}
    p_n(x_n)&\sim {\G((c+1)/3)^2\G((c+2)/3)t^n\over3^{3n}n^{c+1/2}(2\pi/n)^{3/2} 2^{5/6} }(1-108/t)^{-1/4}\left[{1+\sqrt{1-108/t}\over 2}\right]^{2n+1+{(2c-2)\over 3}}
  \nonumber\\&\times\exp\{\int_0^1{3n+c+1/2\over\sqrt{1-108s^3/t}}ds\}.
  \end{align*}
\end{prop}
\begin{proof}
Recall the rates in \eqref{masterII} and $K_n$ defined in \eqref{M2kn-def}. By Stirling's formula, we have $K_n\sim 2^{5/6} n^{-5/6}$ and
  \begin{align*}
    \prod_{k=0}^{n-1}\la_k&={3^{3n}\G((c+1)/3+n)^2\G((c+2)/3+n)\over\G((c+1)/3)^2\G((c+2)/3)}
    \\&\sim{3^{3n}n^{c+4/3}\G(n)^3\over\G((c+1)/3)^2\G((c+2)/3)}
    \\&\sim{3^{3n}n^{c+4/3}(2\pi/n)^{3/2}(n/e)^{3n}\over\G((c+1)/3)^2\G((c+2)/3)}.
  \end{align*}
  Then from \eqref{qn&pin}, we have
  $$\pi_n(x)=K_n\prod_{k=0}^{n-1}\la_kp_n(x)\sim{3^{3n}n^{c+1/2}(2\pi/n)^{3/2}(n/e)^{3n}2^{5/6}\over\G((c+1)/3)^2\G((c+2)/3)}p_n(x).$$
  Because $\pi_n(x)$ satisfies \eqref{pi-outer-general} with $p=3$, $b=27$, $u=c+4/3$ and $v=c-1/3$, we get our proposition from the above formula.
\end{proof}
Again, based on the above proposition, we obtain the following result in the oscillatory region.
\begin{prop}
  Let $x=x_n:=(n+(2c+1)/6)^3t$ with $t=108\cos^2\theta,\theta \in [\delta,\frac{\pi}{2}-\delta]$, we have as $n\to\infty$,
\begin{align} \label{M2-pn-asy}
    p_n(x_n)&\sim {3^{c+1/2} \G^2({c+1 \over 3})\G({c+2 \over 3}) \over 2^{4/3} \pi^{3/2} (108t-t^2)^{1/4}} \frac{1}{(n^3t)^{\frac{c-1}{3}}} \exp\left( (n+{2c+1\over 6}) (\frac{t}{108})^{1/3} \rho \right)
  \nonumber\\&
  \times \cos\{{\pi\over4}-(2n+{(2c+1)\over 3})\theta + (n+{ 2c+1\over 6 }) \int_{t}^{108} \frac{(t/s)^{1/3}}{\sqrt{108s-s^2}} ds \},
\end{align}
where $\rho$ is defined in \eqref{rho-M1}.
\end{prop}
Then the Airy-type expansion can be obtained.
\begin{thm}
  Let $\nu = n+\frac{2c+1}{6}$. With $K_n$ and $U(t)$ defined in \eqref{M2kn-def}, \eqref{M1zeta-def1} and \eqref{M1zeta-def2}, respectively, we have
  \begin{eqnarray*}
    \mathcal{Q}_n(\nu^3 \, t) &\sim&  (-1)^n K_n {3^{c+\frac{1}{2}} \G^2({c+1 \over 3})\G({c+2 \over 3}) \over 2^{\frac{4}{3}} \pi \, \nu^{c-1} \, t^{\frac{c-1}{3}}}  \exp\left(  \frac{\sqrt{\pi} \, \Gamma(\frac{1}{3})}{3\sqrt[3]{4} \, \Gamma(\frac{5}{6})} \nu t^{\frac{1}{3}} \right) \left(\frac{U(t)}{t(t-108) }\right)^{\frac{1}{4}} \nonumber \\
     && \times \left[ \Ai(\nu^{\frac{2}{3}} U(t)) \sum_{s=0}^\infty \frac{\tilde{A}_s(U)}{\nu^{s-\frac{1}{6}}} + \Ai'(\nu^{\frac{2}{3}} U(t)) \sum_{s=0}^\infty \frac{\tilde{B}_s(U)}{\nu^{s+\frac{1}{6}}} \right],
  \end{eqnarray*}
  uniformly for $0<\delta \leq t < \infty$.
\end{thm}

\begin{proof}

When $t<108$, recall the asymptotic expansion for Airy function in \eqref{Airy-asy1} and \eqref{Airy-asy2} again, we have from \eqref{M1-Pn} and \eqref{M1-Qn}
\begin{align*}
  p_n(x) &=  C_1(x) P_n(x) + C_2(x) Q_n(x) \sim C_1(x)  \frac{3\sqrt{6} }{[t (108-t)]^{\frac{1}{4}} }
   \frac{1}{\sqrt{\pi} }  \cos \biggl( \frac{2\nu}{3} (-U(t))^{\frac{3}{2}} -\frac{\pi}{4}\biggr) \nonumber \\
   & \quad - C_2(x)  \frac{3\sqrt{6}}{[t (108-t)]^{\frac{1}{4}} } \frac{1}{\sqrt{\pi} }  \sin \biggl( \frac{2\nu}{3} (-U(t))^{\frac{3}{2}} -\frac{\pi}{4} \biggr) \nonumber \\
   &= C_1(x)  \frac{3\sqrt{6} }{[t (108-t)]^{\frac{1}{4}} }
   \frac{1}{\sqrt{\pi} }  \cos \biggl( \nu( \cos^{-1} \left( \frac{t-54}{54} \right) - \int_t^{108} \frac{(t/s)^{\frac{1}{3}}}{\sqrt{s(108-s)}} ds)  -\frac{\pi}{4}\biggr) \nonumber \\
   &  \quad - C_2(x) \frac{3 \sqrt{6} }{[t (108-t)]^{\frac{1}{4}} } \frac{1}{\sqrt{\pi} }   \sin \biggl( \nu( \cos^{-1} \left( \frac{t-54}{54} \right) - \int_t^{108} \frac{(t/s)^{\frac{1}{3}}}{\sqrt{s(108-s)}} ds)  -\frac{\pi}{4}\biggr). \nonumber
\end{align*}
Comparing the above formula and \eqref{M2-pn-asy}, we have
\begin{equation*}
  C_1(x) = {3^{c-1} \G^2((c+1)/3)\G((c+2)/3) \over 2^{5/6} \pi \, x^{\frac{c-1}{3}}}  \exp\left( (\frac{x}{108})^{\frac{1}{3}} \frac{\sqrt{\pi} \, \Gamma(\frac{1}{3})}{\Gamma(\frac{5}{6})}  \right),  \qquad C_2(x) = 0.
\end{equation*}
This completes the proof of the theorem.
\end{proof}

\begin{rmk}
  Again, we can obtain the limiting zero distribution for $\mathcal{Q}_n(108 n^3 x)$ from the recurrence relation \eqref{masterII}. It is the same as that in \eqref{masterI-measure}.
\end{rmk}

Since the above theorem is uniformly valid in the neighbourhood of the large transition point $108$, we have the following asymptotic formula for $\mathcal{Q}_n(x)$.

\begin{cor} \label{M2-cor1}
  Let $\nu = n+\frac{2c+1}{6}$ and $K_n$ be given in \eqref{M2kn-def}. Uniformly for a bounded real number $s$, we have, for $x = 108 \nu^3 + 54 \sqrt[3]{18} \,s  \,\nu^{\frac{7}{3}}$
  \begin{eqnarray}
    && x^{\frac{c-1}{3}} \exp\left(  -\frac{\sqrt{\pi} \, \Gamma(\frac{1}{3})}{3\sqrt[3]{4} \, \Gamma(\frac{5}{6})} x^{\frac{1}{3}} \right) \mathcal{Q}_n(x) =  (-1)^n K_n {3^{c-\frac{7}{6} } \G^2({c+1 \over 3})\G({c+2 \over 3}) \nu^{\frac{1}{6}} \over 4 \sqrt[6]{2}  \; \pi }    \left[ \Ai(s)  + O(\nu^{-\frac{2}{3}}) \right], \nonumber \\
     && \hspace{12cm} \textrm{as $\nu \to \infty.$} \label{M2-airy}
  \end{eqnarray}
\end{cor}
\begin{proof}
  The proof is very similar to that of Corollary \ref{M1-cor1}.
\end{proof}
\begin{rmk}
  Again we can put the left hand side of \eqref{M2-airy} into its orthonormal form.
  Let $\widehat {\mathcal{Q}}_{n}(x)$ be the orthonormal Conrad--Flajolet polynomials II, then \eqref{M2-airy} can be rewritten as
  \begin{eqnarray}
    x^{\frac{c-1}{3}} \exp\left(  -\frac{\sqrt{\pi} \, \Gamma(\frac{1}{3})}{3\sqrt[3]{4} \, \Gamma(\frac{5}{6})} x^{\frac{1}{3}} \right) \widehat {\mathcal{Q}}_n(x) =  {\G(c+1 ) \nu^{-\frac{4}{3}} \over 2^{\frac{1}{3}}   \; 3^{\frac{5}{3}}  }    \left[ \Ai(s)  + O(\nu^{-\frac{2}{3}}) \right], \ \textrm{as $\nu \to \infty.$} \label{M2-hat-airy}
  \end{eqnarray}
  According to the above formula, again it seems reasonable to conjecture that at least one of the weight functions for the Conrad--Flajolet polynomials II should behave like
  \begin{equation*}
    x^{\frac{2(c-1)}{3}} \exp\left(  -\frac{\sqrt{\pi} \sqrt[3]{2} \, \Gamma(\frac{1}{3})}{3 \, \Gamma(\frac{5}{6})} x^{\frac{1}{3}} \right) \qquad \textrm{as }x \to \infty.
  \end{equation*}
\end{rmk}

Since \eqref{M2-airy} is very similar to \eqref{M1-airy} except for some constants on the right hand side and powers of $x$ changed on the left hand side, we have the same approximation for the zeros as in Corollary \ref{M1-corol}, that is,
  \begin{equation*}
    x_{n,k} = 108 \nu^3 + 54 \sqrt[3]{18}\; \mathfrak{a}_k \nu^{7/3} + O(\nu^{5/3}),
  \end{equation*}
  where $\nu = n + \frac{2c+1}{6}$.

The Bessel-type expansion is also obtained as follows.
\begin{thm}\label{M2-thm2}
  Let $\nu = n+\frac{2c+1}{6}$. With $K_n$ and $U^*(t)$ defined in \eqref{M2kn-def}, \eqref{M1zeta2-def1} and \eqref{M1zeta2-def2}, respectively, we have
  \begin{align*}
    \mathcal{Q}_n(\nu^3 \, t) &\sim  K_n {3^{c+1/2} \G^2({c+1 \over 3})\G({c+2 \over 3}) \; \nu^{\frac{3}{2}-c} \over 2^{11/6} \pi \; t^{\frac{c-1}{3}} }   \exp\left( \frac{\sqrt{\pi} \, \Gamma(\frac{1}{3})}{3\sqrt[3]{4} \, \Gamma(\frac{5}{6})} \nu \, t^{\frac{1}{3}} \right)  \left(\frac{U^*(t)}{t(108-t) }\right)^{\frac{1}{4}} \nonumber \\
     & \hspace{-60pt} \times \biggl\{ \sin\left( \frac{\sqrt{\pi} \, \Gamma(\frac{1}{3})}{2^{\frac{2}{3}} \sqrt{3} \,\Gamma(\frac{5}{6})} \, x^{\frac{1}{3}} -  \frac{c}{3} \pi \right) \left[ J_{\frac{2}{3}}(\nu {U^*}^{\frac{1}{2}}(t)) \sum_{s=0}^\infty \frac{\tilde{A}_s^*(U^*)}{\nu^{s}} + J_{\frac{5}{3}}(\nu {U^*}^{\frac{1}{2}}(t)) \sum_{s=0}^\infty \frac{\tilde{B}_s^*(U^*)}{\nu^{s}} \right] \nonumber \\
      &\hspace{-60pt} - \cos\left( \frac{\sqrt{\pi} \, \Gamma(\frac{1}{3})}{2^{\frac{2}{3}} \sqrt{3} \,\Gamma(\frac{5}{6})} \, x^{\frac{1}{3}} -  \frac{c}{3} \pi \right) \left[ Y_{\frac{2}{3}}(\nu {U^*}^{\frac{1}{2}}(t)) \sum_{s=0}^\infty \frac{\tilde{A}_s^*(U^*)}{\nu^{s}} + Y_{\frac{5}{3}}(\nu {U^*}^{\frac{1}{2}}(t)) \sum_{s=0}^\infty \frac{\tilde{B}_s^*(U^*)}{\nu^{s}} \right] \biggr\},
      \nonumber
  \end{align*}
  uniformly for $-\infty< t \leq M < 108$.
\end{thm}
\begin{proof}
  The proof is similar to that of Theorem \ref{M1-thm2}.
\end{proof}

\section{Remarks on the Moment Problem}
The general solution of an indeterminate moment problem is described in terms of
four entire functions   $A(z), B(z), C(z), D(z)$ such that
\begin{equation*}
A(z)  D(z) -  B(z) C(z)  =1
\end{equation*}
 The functions are called the Nevanlinna parametrization; see \cite{Akh, Sho:Tam}. It is known that these functions have real and simple zeros. The following theorem, see \cite[Thm. 21.1.2]{Ismbook}, describes all solutions to an indeterminate moment problem.
\begin{thm}
\label{thm21.1.2}
Let $\mathcal{N}$ denote the class of functions $\{\sigma\}$, which are analytic in the open upper half plane and map it into
the lower half plane, and satisfy $\sigma(\overline{z})=\overline{\sigma(z)}$. Then the formula
\begin{equation}
\label{eq21.1.11}
\int_{\mathbb R}\frac{d\mu(t;\sigma)}{z-t}
=\frac{A(z)-\sigma(z)\,C(z)}{B(z)-\sigma(z)\,D(z)},\quad z\notin\mathbb{R}
\end{equation}
establishes a one-to-one correspondence between the solutions $\mu$ of the moment problem and functions
$\sigma$ in the class $\mathcal{N}$,  augmented by the constant $\infty$.
\end{thm}
A theorem of Herglotz, \cite[Lemma 2.1]{Sho:Tam} asserts that any $\sigma$ in the above theorem has the representation
\begin{eqnarray}
\label{eqHer}
\sigma(z) = - c_1z + c_2 + \int_{\mathbb{R}} \frac{1+tz}{z-t} d\alpha(t),
\end{eqnarray}
where $\alpha$ is a finite positive measure on $\mathbb{R}$, $c_1$ and $c_2$ are real constants, and
$c_1 \ge 0$.

The next theorem is in   \cite{Ber:Chr} and
\cite{Ism:Mas}.
\begin{thm}\label{thm21.1.4}
Let $\sigma$ in \eqref{eq21.1.11} be analytic in $\im z>0$, and assume $\sigma$ maps $\im z>0$ into $\im\sigma(z)<0$. If
$\mu(x,\sigma)$  does not have a jump at $x$ and $\sigma(x\pm i0)$ exist then
\begin{equation}
\label{eq21.1.12}
\frac{d\mu(x;\sigma)}{dx}=\frac{\sigma(x-i0^+)-\sigma(x+i0^+)}
{2\pi i\;\left| B(x)-\sigma(x-i0^+)D(x)\right|^2}.
\end{equation}
\end{thm}

 A consequence of Theorem \ref{thm21.1.4} is that the polynomials associated with this moment  are orthogonal with respect to the weight function
 \begin{eqnarray}
 \label{eqwind}
 w(x) = \frac{\gamma/\pi}{\gamma^2 B^2(x) + D^2(x)},
 \end{eqnarray}
 for any $\gamma > 0$.

 Berg and Pedersen \cite{Ber:Ped} showed that the four entire functions $A, B, C, D$ have the same order, type, and the Phragm\'{e}n-Lindel\"{o}f indicator.

 The formulas for the orthonormal polynomials in \eqref{fn-hat-airy}, \eqref{bv-hat-airy}, \eqref{M1-hat-airy} and \eqref{M2-hat-airy},  seem  to indicate an interesting pattern about the powers of $\nu$ (a constant plus n, the degree of the polynomial).  Indeed it  suggested to  us to formulate  the following conjecture.
  \begin{conj}
  Let $\{\widehat P_k(x)\}_{k=1}^n$ be a sequence of polynomials orthonormal with respect to the weight function $w(x)$ on an unbounded interval and $t_n$ be the large transition point of $\widehat P_n(x)$. Without loss of generality, we may further assume the interval to be $(0,\infty)$ or $(-\infty, \infty)$. If the weight function $w(x)$ has  the following behavior
  \begin{equation}
  \label{eqm}
    -\log w(x) = O(x^m) \qquad \textrm{as } x \to \infty,
  \end{equation}
  then we have
  \begin{equation*}
    t_n = O(n^\frac{1}{m}), \qquad \textrm{as } n \to \infty.
  \end{equation*}
  Moreover, uniformly for a bounded real number $s$, we have, for $x = t_n(1+sn^{-\frac{2}{3}})$
    \begin{equation*}
    w(x)^{\frac{1}{2}} \widehat P_{n}(x) \sim \hat c(s) \; n^{k}   (1+o(1))  , \qquad \textrm{as $n \to \infty,$}
  \end{equation*}
  where $\hat c(s)$ is a uniformly bounded function in $s$ and the constant $k$ is given by
  \begin{equation} \label{pattern-conj}
    k = \frac{1}{6}-\frac{1}{2m}.
  \end{equation}
  \end{conj}
  Here  $m >0$ but is not necessarily an integer.

  Note that the above conjecture is true for Laguerre polynomials, Hermite polynomials and polynomials orthogonal with respect to the Freud weight $e^{-|x|^\beta}$. For the  asymptotics of Laguerre polynomials and Hermite polynomials, see \cite[Sec. 8.22]{Sze}. For the asymptotics of polynomials orthogonal with respect to the Freud weight, see \cite[Eq. (1.19)]{Kri:McL}.

   When the corresponding moment problem is
  indeterminate, $m$ in \eqref{eqm}  equals the order of any of the entire functions $A, B, C$ or $D$.  This
  seems to be the case for the weight function \eqref{eqwind}.  It also agrees with the known weight functions for the Chen--Ismail polynomials.

\begin{rmk}
  Note that the large transition point $t_n$ is closely related to the largest zero of the orthogonal polynomials. The asymptotics for the largest zero has been studied in the literature, which is related to \eqref{eqm} in our conjecture. For example, Levin and Lubinsky obtained the asymptotics for determinate exponential weights in \cite[Thm. 11.3]{Lev:Lub:book}. They also derived similar results for indeterminate weights and some weights close to indeterminate ones in \cite[Cor. 1.2]{Lev:Lub1995} and \cite[Cor. 1.4]{Lev:Lub2007}, respectively. However, the interesting relation between $k$ and $m$ in \eqref{pattern-conj} seems to be new in the literature.
\end{rmk}

In the rest of  this section, we shall  show that the moment problems corresponding to Conrad-Flajolet polynomials I and II  do  not have unique solutions, that is the moment problems are indeterminate. The following proposition is in \cite{Akh} or \cite{Sho:Tam}.
\begin{prop}
Let $\widehat\P_n(x)$ be the orthonormal polynomials satisfying the following recurrence relation:
\begin{equation}\label{orthonormal}
  x\widehat\P_n(x)=b_{n+1}\widehat\P_{n+1}(x)+a_n\widehat\P_n(x)+b_n\widehat\P_{n-1}(x).
\end{equation}
If the moment problem
has a unique solution then the series
\begin{equation}\label{series}
\sum_{n=1}^\infty |\widehat\P_n(z)|^2
\end{equation}
diverges for $z \notin \mathbb{R}$. If the above
series
diverges at one $z \notin \mathbb{R}$, then the moment problem has a unique solution.
\end{prop}

Recall that a birth and death process  leads to polynomials $\Q_n(x)$ defined in \eqref{bir&dea}.
We have $a_n=\lambda_n+\mu_n$, $b_n=\sqrt{\la_{n-1}\mu_n}$ in \eqref{orthonormal} and the corresponding orthonormal one is
\begin{equation*}
\widehat\Q_n(x)=(-1)^n\Q_n(x)\sqrt{\prod_{k=1}^n{\la_{k-1}\over\mu_k}}.
\end{equation*}
For the orthonormal Conrad-Flajolet polynomials I, we obtain the following result from Theorem \ref{M1-thm2}.
\begin{cor}
  Let $\nu = n+\frac{c+1}{3}$ and $x=\nu^3t$. With $K_n$ and $U^*(t)$ defined in \eqref{M1kn-def}, \eqref{M1zeta2-def1} and \eqref{M1zeta2-def2}, respectively, we have
  \begin{align*}
    &\widehat\Q_n(\nu^3 \, t) \nonumber\\ \sim&  (-1)^n\sqrt{\prod_{k=1}^n{\la_{k-1}\over\mu_k}}K_n {3^{c+1} \G({c+1 \over 3})\G^2({c+2 \over 3}) \; \nu^{1-c} \over 2^{5/3} \pi \; t^{\frac{2c-1}{6}} }   \exp\left( \frac{\sqrt{\pi} \, \Gamma(\frac{1}{3})}{3\sqrt[3]{4} \, \Gamma(\frac{5}{6})} \nu \, t^{\frac{1}{3}} \right)  \left(\frac{U^*(t)}{t(108-t) }\right)^{\frac{1}{4}} \nonumber \\
     \times& \biggl\{ \sin\left( \frac{\sqrt{\pi} \, \Gamma(\frac{1}{3})}{2^{\frac{2}{3}} \sqrt{3} \,\Gamma(\frac{5}{6})} \, x^{\frac{1}{3}} -  \frac{c}{3} \pi \right) \left[ J_{\frac{1}{3}}(\nu {U^*}^{\frac{1}{2}}(t)) \sum_{s=0}^\infty \frac{\tilde{A}_s^*(U^*)}{\nu^{s}} + J_{\frac{4}{3}}(\nu {U^*}^{\frac{1}{2}}(t)) \sum_{s=0}^\infty \frac{\tilde{B}_s^*(U^*)}{\nu^{s}} \right] \nonumber \\
      -& \cos\left( \frac{\sqrt{\pi} \, \Gamma(\frac{1}{3})}{2^{\frac{2}{3}} \sqrt{3} \,\Gamma(\frac{5}{6})} \, x^{\frac{1}{3}} -  \frac{c}{3} \pi \right) \left[ Y_{\frac{1}{3}}(\nu {U^*}^{\frac{1}{2}}(t)) \sum_{s=0}^\infty \frac{\tilde{A}_s^*(U^*)}{\nu^{s}} + Y_{\frac{4}{3}}(\nu {U^*}^{\frac{1}{2}}(t)) \sum_{s=0}^\infty \frac{\tilde{B}_s^*(U^*)}{\nu^{s}} \right] \biggr\},
      \nonumber
  \end{align*}
  uniformly for $-\infty< t \leq M < 108$.   Here $\la_n$ and $\mu_n$ are defined in \eqref{masterI}.
\end{cor}
Let $z=\nu^3t$ be a fixed complex number, then $t=z/\nu^3=O(n^{-3})$. It follows from \eqref{M1-U*} that $U^*(t)=O(t)=O(n^{-3})$. Recall the asymptotic formulas of Bessel functions at small arguments, we obtain
\begin{align*}
  J_{\frac{1}{3}}(\nu {U^*}^{\frac{1}{2}}(t))&=O(n^{-1/6}),~J_{\frac{4}{3}}(\nu {U^*}^{\frac{1}{2}}(t))=O(n^{-2/3}),\\
  Y_{\frac{1}{3}}(\nu {U^*}^{\frac{1}{2}}(t))&=O(n^{1/6}),~~~Y_{\frac{4}{3}}(\nu {U^*}^{\frac{1}{2}}(t))=O(n^{2/3}).
\end{align*}
Since $\tilde{A}_s^*(U^*)=1$ and $\tilde{B}_s^*(U^*)=0$ by \eqref{M1-AB}, we observe that the expressions involving Bessel functions and trigonometric functions are of order $O(n^{1/6})$.
Next, we obtain the following asymptotic estimate  from \eqref{masterI}, \eqref{M1kn-def} and Striling's formula,
$$\sqrt{\prod_{k=1}^n{\la_{k-1}\over\mu_k}}=O(n^{-5/6}),~ K_n=O(n^{-2/3}).$$
A combination of the above estimates yields $\widehat\Q_n(z)=O(n^{-5/6})$. This
establishes   the convergence  of the series  in \eqref{series}. Thus, the corresponding moment problem is indeterminate.

For the orthonormal Conrad-Flajolet polynomials II, we obtain the following result from Theorem \ref{M2-thm2}.
\begin{cor}
  Let $\nu = n+\frac{2c+1}{6}$. With $K_n$ and $U^*(t)$ defined in \eqref{M2kn-def},  \eqref{M1zeta2-def1} and \eqref{M1zeta2-def2}, respectively, we have
  \begin{align*}
    &\widehat\Q_n(\nu^3 \, t)
    \nonumber\\ \sim &  (-1)^n\sqrt{\prod_{k=1}^n{\la_{k-1}\over\mu_k}}K_n {3^{c+1/2}
     \G^2({c+1 \over 3})\G({c+2 \over 3}) \; \nu^{\frac{3}{2}-c} \over 2^{11/6} \pi \; t^{\frac{c-1}{3}} }   \exp\left( \frac{\sqrt{\pi} \, \Gamma(\frac{1}{3})}{3\sqrt[3]{4} \, \Gamma(\frac{5}{6})} \nu \, t^{\frac{1}{3}} \right)  \left(\frac{U^*(t)}{t(108-t) }\right)^{\frac{1}{4}} \nonumber \\
     \times& \biggl\{ \sin\left( \frac{\sqrt{\pi} \, \Gamma(\frac{1}{3})}{2^{\frac{2}{3}} \sqrt{3} \,\Gamma(\frac{5}{6})} \, x^{\frac{1}{3}} -  \frac{c}{3} \pi \right) \left[ J_{\frac{2}{3}}(\nu {U^*}^{\frac{1}{2}}(t)) \sum_{s=0}^\infty \frac{\tilde{A}_s^*(U^*)}{\nu^{s}} + J_{\frac{5}{3}}(\nu {U^*}^{\frac{1}{2}}(t)) \sum_{s=0}^\infty \frac{\tilde{B}_s^*(U^*)}{\nu^{s}} \right] \nonumber \\
      -& \cos\left( \frac{\sqrt{\pi} \, \Gamma(\frac{1}{3})}{2^{\frac{2}{3}} \sqrt{3} \,\Gamma(\frac{5}{6})} \, x^{\frac{1}{3}} -  \frac{c}{3} \pi \right) \left[ Y_{\frac{2}{3}}(\nu {U^*}^{\frac{1}{2}}(t)) \sum_{s=0}^\infty \frac{\tilde{A}_s^*(U^*)}{\nu^{s}} + Y_{\frac{5}{3}}(\nu {U^*}^{\frac{1}{2}}(t)) \sum_{s=0}^\infty \frac{\tilde{B}_s^*(U^*)}{\nu^{s}} \right] \biggr\},
  \end{align*}
  uniformly for $-\infty< t \leq M < 108$. Here $\la_n$ and $\mu_n$ are defined in \eqref{masterII}.
\end{cor}
A similar argument as in the case of Conrad-Flajolet polynomials I, we conclude that the moment problem for Conrad-Flajolet polynomials II is also indeterminate.

\subsection*{Acknowledgements}
  We would like to thank Professor Doron Lubinsky for his helpful comments and discussions. We are also grateful to the editors and referees for their valuable suggestions and extremely careful reading of the manuscript. Dan Dai was partially supported by a grant from City University of Hong Kong (Project No. 7002883) and a grant from the Research Grants Council of the Hong Kong Special Administrative Region, China (Project No. CityU 100910). Mourad E.H. Ismail was partially supported by NPST Program of King Saud University (Project No. 10-MAT1293-02) and by the DSFP at King Saud University in Riyadh, and by a grant from the Research Grants Council of the Hong Kong Special Administrative Region, China (Project No. CityU 1014111).


\end{document}